\documentclass[a4paper,12pt]{article}
\usepackage[english]{babel}
\usepackage[tbtags]{amsmath}
\usepackage{amssymb}
\usepackage{amsfonts}
\usepackage{amsthm}
\usepackage{calrsfs}
\usepackage{array}
 \usepackage{bbm}
\usepackage{stmaryrd}
\usepackage{upgreek}
\usepackage[svgnames]{xcolor}
\usepackage{hyperref}
\usepackage{mathtools}
\def\div{\, \mbox{div}\,  }

\def\w{\rho (y)}
\def\p{\phi }
\def\ps{\phi(s)}
\numberwithin{equation}{section}
\def\ibint {\int_{B}}
\def\iint {\int_{-1}^1}

\def\grad{\partial_y}
\def\ia{\int_{s}^{s+1}}
\def \I {B(x_0,T(x_0)-t)}

\def\T{T}
\def\v{{\mathrm{d}}v}

\def\Box{\hfill\rule{2.5mm}{2.5mm}}
\def\t{{\mathrm{d}}\tau}
\def\s{{\mathrm{d}}s}
\def\y{{\mathrm{d}}y}

\def\no{\nonumber}

\def\grad{\nabla}

\def\m {{m}}

\def \er {\mathbb R}

\def\R{{\mathbb {R}}}

\def\cprime{$'$}

\newcommand{\ds}{\displaystyle}

\def\J{\frac{1}{s}}

\def\t{{\mathrm{d}}\tau}
\def\s{{\mathrm{d}}s}
\def\y{{\mathrm{d}}y}

\def\y{{\mathrm{d}}y}
\def\N{\int_{s_1}^{s_2}\iint\! e^{-\frac{2(p+1)s}{p-1}}s^{\frac{2a}{p-1}} F(\p w)\w{\mathrm{d}}y{\mathrm{d}}s}

\theoremstyle{plain}
\newtheorem{thm}{Theorem}
\newtheorem*{thm*}{Theorem}
\newtheorem{coro}[thm]{Corollary}

\newtheorem{prop}{Proposition}[section]
\newtheorem{cor}[prop]{Corollary}
\newtheorem{lem}[prop]{Lemma}

\theoremstyle{definition}

\theoremstyle{remark}
\newtheorem{nb}{Remark}[section]

\makeatletter
\def\blfootnote{\xdef\@thefnmark{}\@footnotetext}
\makeatother

\title{\bf The blow-up rate for    a non-scaling invariant   semilinear wave equations }

\author{Mohamed  Ali  Hamza\\
{\it \small 
Imam Abdulrahman Bin Faisal University
P.O. Box 1982 Dammam, Saudi Arabia}\\
Hatem Zaag\\
{\it \small Universit\'e Paris 13, Sorbonne Paris Cit\'e},\\
{\it \small LAGA, CNRS (UMR 7539), F-93430, Villetaneuse, France}
}

\date{}

\allowdisplaybreaks

\begin{document}

\maketitle

\begin{abstract}
 We consider  the    semilinear wave
equation $$\partial_t^2 u -\Delta u =f(u), \quad  (x,t)\in \er^N\times [0,T),\qquad (1)$$ with
$f(u)=|u|^{p-1}u\log^a (2+u^2)$,  where $p>1$ and $a\in \er$.
We show an upper bound for any 
 blow-up solution  of (1).  Then, in  the one space dimensional case,  using this estimate and  the logarithmic property, we prove  that   the exact blow-up rate of any singular solution of (1)  is  given by the ODE solution  associated with $(1)$,
namely $u'' =|u|^{p-1}u\log^a (2+u^2)$. 
  Unlike the pure power  case ($g(u)=|u|^{p-1}u$) the difficulties here  are due to the fact that  equation (1) is not scale  invariant.
\end{abstract}

\medskip



\section{Introduction}
This paper is devoted to the study of blow-up solutions for the
following semilinear  wave equation:
\begin{equation}\label{gen}
\left\{
\begin{array}{l}
\partial_t^2 u =\Delta u +f(u),\qquad  (x,t)\in \er^N\times [0,T),\\
\\
u(x,0)=u_0(x)\in  H^{1}_{loc,u}(\er^N),\qquad 
\partial_tu(x,0)=u_1(x)\in  L^{2}_{loc,u}(\er^N),
\end{array}
\right.
\end{equation}
where $u(t):x\in{\er^N} \rightarrow u(x,t)\in{\er}$ 
with focusing nonlinearity $f$ defined by:
\begin{equation}\label{deff}
 f(u)=   |u|^{p-1}u\log^a (2+u^2), \quad  p>1,\quad  a\in \er.
 \end{equation}
The spaces  $L^{2}_{loc,u}(\er^N)$ and $H^{1}_{loc,u}(\er^N)$ are  defined by
\begin{equation*}
L^{2}_{loc,u}(\er^N)=\{u:\er^N\rightarrow \er/ \sup_{d\in \er^N}(\int_{|x-d|\le 1}|u(x)|^2dx)<+\infty \},
\end{equation*}
and
\begin{equation*}
H^{1}_{loc,u}(\er^N)=\{u\in L^{2}_{loc,u}(\er^N),|\grad u|\in L^{2}_{loc,u}(\er^N) \}.
\end{equation*}
We assume in addition  that $p>1$  and if $N\ge 2$,  we further  assume that
\begin{equation}\label{subc}
p<p_c\equiv 1+\frac{4}{N-1}.
\end{equation}

\medskip

A  semilinear wave equation with   nonlinearity,  including a  logarithmic factor,  has been introduced
in  various nonlinear physical models in the context of nuclear physics,  wave mechanics, optics,  geophysics  etc ... see e.g. \cite{Bia75, Bia76}.

\medskip

  The defocusing case   has been  studied in the mathematical  literature and the first results
are due to
 \cite{T1} where Tao   proved a global well-posedness and scattering 
  result for   the three dimensional  nonlinear wave equation 
 $\partial_{t}^2 u =\Delta u-|u|^{4}u\log (2+u^2),$
in  the radial case. See also the 
work of Shih \cite{S1}, where the method is refined to treat   $\partial_{t}^2 u =\Delta u-|u|^{4}u\log^c (2+u^2),$ for any $c\in (0,\frac43)$. Later,  Roy   extends  in \cite{R1} the   results (global well-posedness and scattering)
to solutions to the log-log-supercritical equation
 $\partial_{t}^2 u =\Delta u-|u|^{4}u\log^{c} \big(\log (10+u^2)\big),$
for $c$ small,  without any radial assumption. 
This series of works should be considered as  a starting point  for the   understanding of the global behavior of the solutions in  the Sobolev supercritical regime   $\partial_{t}^2 u =\Delta u-|u|^{p}u,$ where  $p>4$. In this direction, we aim  to give  a light in the understanding of 
the superconformal range ( $p>p_c$) related to the blow-up rate of the  solution of  equation \eqref{NLW} below.

\medskip

Let us mention that  the blow-up question for  the   semilinear heat equation 
$ \partial_{t} u=\Delta u+|u|^{p-1}u\log ^a(2+u^2) $ 
 is  studied  by Duong-Nguyen-Zaag  in \cite{DVZ}. More precisely, they   construct for this equation a solution 
which blows up in finite time $T$,  only at one blow-up point $a$, according to the following asymptotic dynamics:
\begin{equation}\label{heatequiv}
u(x,t)\sim \phi(t)\Big(1+\frac{(p-1)|x-a|^2}{4p(T-t)|\log (T-t)|}\Big)^{-\frac{1}{p-1}},\qquad  as\  t\to T,
\end{equation}
  where $\phi(t)$ is is the unique positive solution of the ODE 
\begin{equation}\label{heat1}
\phi'=|\phi|^{p-1}\phi\log ^a(2+\phi^2),\ \qquad\qquad  \lim_{t
\to T}\phi(t)=+\infty.
\end{equation}
Given that we have the same expression in the pure power nonlinearity case ($g(u)=|u|^{p-1}u$)   with   $\phi(t)$ replaced by 
 $\kappa (T-t)^{-\frac1{p-1}}$  (see \cite{BKnonl94}), we see that the effect of the nonlinearity is all encapsulated in the ODE \eqref{heat1}.

\bigskip

Equation  $\eqref{gen}$ is well-posed in
$H^{1}_{loc,u}\times L^{2}_{loc,u}$. This follows from the finite
speed of propagation and the well-posedness in $H^{1}(\er^N)\times L^{2}(\er^N)$.
The existence of
blow-up solutions $u(t)$ of $\eqref{gen}$ follows from
 ODE techniques or the energy-based blow-up criterion by Levine \cite{Ltams74} (see
also   \cite{LT, STV,T2}).  More
blow-up results can be found in Caffarelli and Friedman
 \cite{CFarma85,  CFtams86}, Kichenassamy and Littman
 \cite{KL1cpde93, KL2cpde93}.
Numerical simulations of blow-up are given by Bizo\'n {\it and al.}  (see \cite{Bjnmp01, BBMWnonl10, BCTnonl04, 
BZnonl09}). 

\medskip

If $u$ is an arbitrary blow-up solution of   \eqref{gen}, we define (see for example Alinhac \cite{Apndeta95}) a 1-Lipschitz curve $\Gamma=\{(x,T(x))\}$
such that the maximal influence domain $D$ of $u$ (or the domain of definition of $u$) is written as
\begin{equation}\label{defdu}
D=\{(x,t)\;|\; t< T(x)\}.
\end{equation}
$\bar T=\inf_{x\in {\R^N}}T(x)$ and $\Gamma$ are called the blow-up time and the blow-up graph of $u$.
A point $x_0$ is a non characteristic point
if  there are
\begin{equation}\label{nonchar}
\delta_0\in(0,1)\mbox{ and }t_0<T(x_0)\mbox{ such that }
u\;\;\mbox{is defined on }{\mathcal C}_{x_0, T(x_0), \delta_0}\cap \{t\ge t_0\}
\end{equation}
where ${\cal C}_{\bar x, \bar t, \bar \delta}=\{(x,t)\;|\; t< \bar t-\bar \delta|x-\bar x|\}$.

\bigskip

In this paper, we study the  blow-up rate of any singular solution of \eqref{gen}. Before going on, it is necessary
to mention  that
 the blow-up rate in the case
with pure  power nonlinearity 
\begin{equation}\label{NLW}
\partial^2_{t} u =\Delta  u+|u|^{p-1}u,   ,\,\,\,(x,t)\in \R^N \times [0,T),
\end{equation}
was studied by Merle and Zaag  in \cite{MZajm03, MZimrn05, MZma05}. More precisely, 
they proved  that if $u$  is a solution of \eqref{NLW}
with blow-up graph $\Gamma : \{ x\mapsto T(x)\}$ and $x_{0}$ is a
non-characteristic point, then,  for all $t\in
[\frac{3T(x_{0})}{4},T(x_{0})]$,
 \begin{eqnarray}\label{mzmz}
  0 < \varepsilon_{0}(p)\leq (T(x_{0})-t)^{\frac{2}{
   p-1}}\frac{\|u(t)\|_{L^{2}(\I )}}
{(T(x_{0})-t)^{\frac{N}2}}\\
  +(T(x_{0})-t)^{\frac{2}{ p-1}+1}\Big(\frac{\|\partial_{t} u(t)\|_{L^{2}(\I)}}
  {(T(x_{0})-t)^{\frac{N}2}}
  + \frac{\|\partial_x u(t)\|_{L^{2}(\I )}}{{(T(x_{0})-t)^{\frac{N}2}}}\Big)\leq K,\nonumber
  \end{eqnarray}
 where the constant $K$ depends only on  $p$ and on an upper bound on
 $T(x_{0})$, ${1}/{T(x_{0})}$, $\delta_{0}(x_{0})$ and the initial data
 in $H^{1}_{loc,u}(\R^N)\times L^{2}_{loc,u}(\R^N)$. Namely, the blow-up rate of any singular solution of \eqref{NLW} is given by the solution of the  associated ODE
$u'' =|u|^{p-1}u$.
 Note that this result about the blow-up rate  is valid   
 in the subconformal  and conformal case  ($1<p\le  p_c$).

\medskip

 In a series of papers,   Merle and Zaag \cite{MZjfa07, MZcmp08, MZajm11, MZisol10}
  (see also   C\^ote and Zaag \cite{CZcpam13}) give a full picture of  blow-up
  for solutions  of  equation \eqref{NLW} in one space dimension.
Among other results, Merle and Zaag proved that characteristic
points are isolated and that the blow-up set $\{(x,T(x))\}$ is
${\cal {C}}^1$ near non-characteristic points and corner-shaped near
characteristic points. 
In higher dimensions, the method used in  the one-dimensional case 
does not remain valid
 because there  is no   classification of selfsimilar solutions of equation \eqref{NLW} in the energy space. However, in the radial case outside the origin, Merle and Zaag  reduce to  the one-dimensional case with  perturbation and  obtain the same results as for $N=1$ (see \cite{MZbsm11}
and also
the extension by Hamza and Zaag in \cite{HZkg12}
to the Klein-Gordon equation and other
damped lower-order perturbations of equation
 \eqref{NLW}). Later, Merle and Zaag could address the higher dimensional case in the subconformal
case and prove the stability of the explicit selfsimilar solution  with respect to the blow-up point
and initial data (see 
\cite{MZods, MZods14}).
Considering the behavior of radial solutions at the origin, Donninger and Sch{\"o}rkhuber were able to prove the stability of the ODE solution $u(t) =\kappa_0 (p)(T - t)^{-\frac{2}{p-1}} $  in the lightcone
 with respect to small perturbations in initial data, in
 a stronger topology (see \cite{DSdpde12,DStams14, DScmp16, DSaihp17}). 
 Their approach is based in particular on a good understanding of the spectral properties of the linearized operator in self-similar variables, operator which is not self-adjoint. Recently, by  
 establishing a suitable Strichartz estimates for the critical wave equation    in similarity variables, 
 Donninger in \cite{DDuke}
 prove the stability of the  solution of the ODE
with respect to small perturbations in initial data, in the energy space.
Let us also mention that 
Killip, Stoval and Vi\c san proved in \cite{KSVma14}
that in  superconformal and Sobolev subcritical range, an upper bound on the blow-up rate is available.
This was further refined by Hamza and Zaag in \cite{HZdcds13}.

\medskip

In \cite{HZjhde12, HZnonl12}, using a highly non-trivial perturbative method, we could obtain the blow-up rate for the
 Klein-Gordon equation  and more generally, for equation 
\begin{equation}\label{NLWP}
\partial_t^2 u =\Delta u+|u|^{p-1}u+f(u)+g(\partial_t u ),\,\,\,(x,t)\in \R^N \times [0,T),
 \end{equation}
under the assumptions  $|f(u)|\leq M(1+|u|^q)$ and $|g(v)|\leq M(1+|v|)$,   for some $M > 0$ and $q<p\le  \frac{N+3}{N-1}$. 
In fact, we proved a similar result to $\eqref{mzmz}$,
valid in the subconformal  and conformal case.
  Let us also mention that in  \cite{H1, omar1, omar2}, the
results obtained in  \cite{HZjhde12,HZnonl12}   were extended to the  
strongly perturbed equation  (\ref{NLWP})  with  $|f(u)|\leq M(1+|u|^p\log^{-a}(2+u^2))$,  for some $a > 1$, though keeping the  same condition in $g$. 

\medskip

%
%
%

In  the previous works   \cite{H1, omar1, omar2,HZjhde12,HZnonl12}, we   consider a  class of
 perturbed equations   where the nonlinear term is  equivalent to the pure power $|u|^{p-1}u$   and we obtain  the  estimate 
 \eqref{mzmz}. This is due to the fact that the dynamics is governed by the ODE equation: $u'' =|u|^{p-1}u$.
Furthermore, our proof remains (non trivially) perturbative  with respect to the homogeneous PDE \eqref{NLW}, which is scale invariant.

\bigskip
This leaves unanswered an interesting question: is the scale invariance  property crucial  in deriving the blow-up rate?

\bigskip
 
In fact we {\textit{had the impression}} that the answer was ''yes'', since the scaling invariance induces in  similarity  variables a PDE which is  autonomous in the unperturbed case \eqref{NLW}, and asymptotically autonomous in the perturbed case \eqref{NLWP}.

\bigskip

In this paper we {\textit{prove}} that the answer is ''no'' from the example  on the non  homogeneous PDE \eqref{NLW}. In fact, 
our situation is different from  \eqref{NLW}, and  \eqref{NLWP}.
 Indeed,   the  term like  $|u|^{p-1}u\log^a (2+u^2)$  
 is playing a fundamental role in the dynamics of  the blow-up  solution of \eqref{gen}. More  precisely, we obtain an analogous  result to \eqref{mzmz} but with a logarithmic correction  as shown in \eqref{main1} below.  In fact, the bow-up rate  is
 given by the solution of the  following ordinary differential equation: $u'' =|u|^{p-1}u\log^a (2+u^2)$.

%
%
%
%
%

\medskip
%

Before handling the PDE, we first study the associated ODE  to \eqref{gen}
\begin{equation}\label{v}
v_T'' (t)=  |v_T (t)|^{p-1}v_T(t) \log^{a}\big(v_T^2(t)  +2\big), \quad v(T)=\infty,
\end{equation}
and show  that  the nonlinear term including  the  logarithmic factor   gives raise to a different
dynamic. In fact, thanks to Lemma \ref{asym-psi-T}, 
 we can see that  the solution $v_T$   satisfies
 \begin{equation}\label{equivv}
 \psi_T(t) \sim \kappa_{a}\psi_T(t), \text{ as } t \to\ T,\quad \textrm{ where}\quad  \kappa_{a} =  \left(\frac{2^{1-2a}(p+1)}{(p-1)^{2-a}} \right)^{\frac1{p-1}},
 \end{equation}
and
 \begin{equation}\label{psi}
 \psi_T(t)=(T - t)^{-\frac{2}{p-1}} (-\log (T - t))^{-\frac{a}{p-1}}.
 \end{equation}

\bigskip
%
Let us introduce the following similarity variables,   defined for all
  $x_0\in \er$, $T_0$ such that $0< T_0\le T(x_0)$ by:
\begin{equation}\label{scaling}
y=\frac{x-x_0}{T_0-t},\quad s=-\log (T_0-t),\quad u(x,t)=\psi_{T_0}(t)w_{x_0,T_0}(y,s).
\end{equation}
From (\ref{gen}), the  function $w_{x_0,T_0}$  (we write $w$ for
simplicity) satisfies the following equation for all $y\in B$, $s>0$ and $s\ge  -\log T_0$:
\begin{align}\label{A}
\partial_{s}^2w&=\frac{1}{\rho}\div(\rho \grad w-\rho(y.\grad w)y)+\frac{2a}{(p-1)s}y.\grad w-\frac{2p+2}{(p-1)^2}w+\gamma(s)w\nonumber\\
&-\Big(\frac{p+3}{p-1}-\frac{2a}{(p-1)s}\Big)\partial_s w
-2y.\grad \partial_{s}w+
e^{-\frac{2ps}{p-1}}s^{\frac{a}{p-1}} f(\ps w),
\end{align}
where   
$\rho (y)=(1-|y|^2)^{\alpha}$,
\begin{equation}\label{alpha}
\alpha=\frac{2}{p-1}-\frac{N-1}{2}>0,
\end{equation}
\begin{equation}
   \gamma(s)=\frac{a(p+5)}{(p-1)^2s}-\frac{a(p+a-1)}{(p-1)^2s^2},\label{defgamma}
\end{equation}
and 
\begin{equation}
\ps =e^{\frac{2s}{p-1}}s^{-\frac{a}{p-1}}.\label{defphi}
\end{equation}

\medskip

 This change of variables  is associated to  the nonlinear wave equation including a  logarithmic nonlinearity \eqref{gen}.
In fact, we have the same transformation as in the 
 pure power  case ($g(u)=|u|^{p-1}u$).
In the new
set of variables $(y,s),$ the behavior of $u$ as $t \rightarrow T_0$
is equivalent to the behavior of $w$ as $s \rightarrow +\infty$.  Also, if $T_0=T(x_0)$, then we simply write $w_{x_0}$
instead of $w_{x_0,T(x_0)}$.

\medskip

 The equation (\ref{A}) will be studied in the Hilbert  space $\cal H$
$${\cal H}=\Big \{(w_1,w_2), |
\displaystyle\ibint\Big(\big ( w_1^2 +|\grad w_1|^2-|y.\grad w_1|^2\big)+w_2^2\Big) \y<+\infty \Big \},$$
where $B=B(0,1)$ stands for the unit ball of $\er^N$
 and throughout the paper.

\medskip

Throughout this paper,
$C$  denotes a  generic positive constant
 depending only on $p,N$  and $a,$  which may vary from line to line. Also,  we will use $K$ to denote a  generic positive constant
 depending only on $p,N, a, \delta_0(x_0)$  and initial data   which may vary from line to line.  We write $f(s)\sim g(s)$ to indicate 
$\displaystyle{\lim_{|s|\to \infty}\frac{f(s)}{g(s)}=1}$. Furthermore, we
denote by
\begin{equation}\label{defF}
 F(u)=\int_{0}^{u}f(v){\mathrm{d}}v=\int_0^u|v|^{p-1}v \log^{{a}}(v^2  + 2)\v.
\end{equation}

\medskip

As we  mentioned earlier,  the invariance of equation \eqref{NLW} under  the   scaling transformation
$u \mapsto u_{\lambda}(x,t)=\lambda^{\frac2{p-1}}u(\lambda x,\lambda t)$   was  crucial
in  the construction of the Lyapunov functional in similarity variables (see Antonini and Merle \cite{AM}).  The  fact that  the
 equation \eqref{NLWP} is not invariant under the last scaling transformation  implies 
that the existence of a Lyapunov functional in similarity variables is far from being trivial (see   \cite{H1, omar1, omar2, HZjhde12,HZnonl12}).
  
\medskip

 In this paper, we prove a   polynomial  (in $s$) space-time bound  on  the similarity variables' version   of the solution $u$ of  \eqref{gen}, valid  in any dimensions  in the subconformal case. 
However, our main contribution  lays,  in one space dimension. It consists in the construction of a 
  Lyapunov  functional in similarity variables  for the problem \eqref{A}  and the proof  that   the blow-up rate of any singular solution of \eqref{gen}  is given by  the  solution of the following ODE:
$u'' =|u|^{p-1}u\log^a (2+u^2)$.

\bigskip

Let us give some details regarding our strategy in this paper.

\bigskip

First, we exploit  some functional  to obtain  a rough estimate on the blow-up solution; namely a polynomial (in $s$) bound on the solution  in similarity variables.  The issue  is how to handle the perturbative terms in \eqref{A}.  In fact,
in order to control them, we  view  equation \eqref{A}  as a perturbation of the
 case of a pure power nonlinearity (case where $a=0$ in \eqref{A}) 
  with the following  terms: 
\begin{equation}\label{perb4}
\frac{2a}{(p-1)s}y.\grad w, \quad \gamma(s)w,\quad   \frac{2a}{(p-1)s}\partial_s w\quad \textrm{ and} \quad  e^{-\frac{2ps}{p-1}}s^{\frac{a}{p-1}} f(\ps w).
\end{equation} 
 The first three  terms 
are    lower order terms which were already  handled  in the subconformal perturbative case treated in
   \cite{HZnonl12,omar1}.  
 However, since  the nonlinear term  
$e^{-\frac{2ps}{p-1}}s^{\frac{a}{p-1}} f(\ps w)$ depends on  time $s$,
  we expect the time derivatives to be delicate. 
 Thanks to the fact that   $uf(u)-(p+1)\int_0^uf(v)
{\mathrm{d}}v\sim  
\frac{2a}{p+1}|u|^{p+1}\log^{a-1}(2+u^2)$, as $u\to \infty$,
we   construct a functional    (in Section \ref{section2})
  satisfying this kind of differential inequality:
 \begin{equation}\label{jardin1}
\frac{d}{ds}h(s)\le -\alpha \ibint (\partial_{s}w)^2\frac{\w}{1-|y|^2}\y
+ \frac{C}{s}h(s), 
\end{equation} 
where $\alpha$ is defined in \eqref{alpha}, and this   implies  a
polynomial estimate.

\medskip

Now, we announce the following rough   polynomial space-time estimate:

 \begin{thm}\label{tp1}
\noindent  Consider $u $   a solution of ({\ref{gen}}) with
blow-up graph $\Gamma:\{x\mapsto T(x)\}$ and  $x_0$  a non
characteristic point. Then,  there exists
$t_{0}(x_{0})\in [ 0,T(x_{0}))$ and  $q=q(p,a,N)>0$  such that,   for all $T_0 \in (t_{0}(x_0),T(x_{0})]$,  for all $s\geq -\log (T_0-t_{0}(x_{0}))$,
we have
\begin{equation}\label{feb19}
\int_{s}^{s+1}\!\ibint \big( w^2(y,\tau )+(\partial_{s}w(y,\tau  ))^2+|\grad w(y,\tau )|^2\big)\y \t  \leq K_1 s^q,
\end{equation}
where $w=w_{x_0,T_0}$, 
 $K_1$ depends on $ p, a,  \delta_{0}(x_{0})$, 
 $T(x_{0})$, $t_0(x_0)$ and\\ 
$\|(u(t_0(x_0)),\partial_tu(t_0(x_0)))\|_{H^{1}\times
L^{2}(B(x_0,\frac{T(x_0)-t_0(x_0)}{\delta_0(x_0)}) )}$.
\end{thm}

\medskip

In the original variables,  Theorem \ref{tp1} implies the following:
\begin{coro}\label{cor1}
\noindent  Consider $u $   a solution of ({\ref{gen}}) with
blow-up graph $\Gamma:\{x\mapsto T(x)\}$ and  $x_0$  a non
characteristic point. Then,  there exists
$t_{0}(x_{0})\in [ 0,T(x_{0}))$ and  $q=q(p,a,N)>0$  such that,   for all $t \in [t_{0}(x_0),T(x_{0}))$,  we have
 $$\int_{t}^{T(x_{0})-\frac1{e}(T(x_0)-t)} \!\!\int_{B(x_{0},T(x_{0})-\tau)} \frac{u^2(x,\tau)}{\psi_{T(x_0)}(\tau)(T(x_0)-\tau)^{\frac{N}2} }{\mathrm{d}}x{\mathrm{d}}\tau  \leq   K_{2}\Big(-\log(T(x_{0})-t)\Big)^{q},$$
and
 $$\int_{t}^{T(x_{0})-\frac1{e}(T(x_0)-t)} \!\!\int_{B(x_{0},T(x_{0})-\tau)} \frac{|\nabla u(x,\tau)|^{2}+(\partial_tu(x,\tau))^2}{\psi_{T(x_0)}(\tau)(T(x_0)-\tau )^{\frac{N}2-1}}{\mathrm{d}}x{\mathrm{d}}\tau  \leq   K_{2}\Big(-\log(T(x_{0})-t)\Big)^{q}.$$
\end{coro}

\begin{nb}
The estimates obtained in Theorem   \ref{tp1} and  Corollary  \ref{cor1}    do not seem
to be optimal unfortunately.  Indeed, we expect the solution of the PDE $u$ to be bounded by the solution of the ODE $\psi_{T(x_0)}$, as in the  case $a=0$. Accordingly, we conjecture that the righ-hand sides in the inequalities in Theorem \ref{tp1}
and Corolllary \ref{cor1} to be constant.
\end{nb}

\bigskip

 Even though the rough estimate obtained  seems bad, it is very useful to
 allow us to derive, in one space dimension,  a  Lyapunov functional for equation \eqref{A}. More precisely, we use this  polynomial estimate and the structure of the
 nonlinear term  to construct    a  Lyapunov functional for equation \eqref{A}  as  a  crucial step to derive  the optimal estimate. 
Let us note that  
the method is valid  only in one dimensional case
and breaks down  in higher dimensional case (see below in Remark \ref{1.7}). For that reason,   Theorem \ref{t1} and Theorem \ref{t2}  given below  are  valid  only in  the  one dimensional case.
Accordingly,
in the rest of this paper, we  consider the  one dimensional case.

\bigskip

To state our main result,  we start by introducing the following 
functionals,
\begin{eqnarray}
E_{1}(w(s),s)\!\!\!&=&\!\!\!\!\iint \Big(\frac{1}{2}(\partial_{s}w)^2+\frac{1}{2}(\partial_y w)^2(1-y^2)+\frac{p+1}{(p-1)^2}w^2-e^{-\frac{2(p+1)s}{p-1}}s^{\frac{2a}{p-1}}   F(\p w)\Big)\w \y, \nonumber\\
&&\no\\
L_0(w(s),s)&=&E_{1}(w(s),s)-\frac1{s\sqrt{s}}
\iint \partial_{s}ww\w\y,\label{5jan1}
\end{eqnarray}
where $F$ is defined by  \eqref{defF}.
Moreover,  for all
 $s\ge  \max (1, -\log T_0)$,  we define the functional
\begin{equation}\label{10dec2}
L(w(s),s)=\exp\Big(\frac{p+3}{\sqrt{s}}\Big) L_0(w(s),s)+\theta e^{-s},
\end{equation}
where $\theta$ is a sufficiently large constant that will be determined later.
We derive that the functional  $L(w(s),s)$ is a decreasing 
  functional  of time  for equation (\ref{A}),  provided that $s$ is  large enough.
Clearly, by  \eqref{5jan1} and \eqref{10dec2},  the  functional $L(w(s),s)$  is a small perturbation of the   natural energy $E_{1}(w(s),s)$.

\medskip

Here is  the statement of  our main theorem in this paper.
 \begin{thm}\label{t1}
Consider   $u $    a solution of ({\ref{gen}}) in one space dimension ($N=1$),  with blow-up graph
$\Gamma:\{x\mapsto T(x)\}$, and  $x_0$  a non characteristic point.
Then there exists $t_1(x_0)\in [0,T(x_0)) $ such that, 
 for all $T_0\in  (t_1(x_0),T(x_0)]$,  for all  $s\ge  -\log(T_0-t_1(x_0))$, we have
\begin{equation}\label{t1lyap}
 L(w(s+1),s+1)-L(w(s),s) \leq -\frac{2}{p-1} \int_{s}^{s+1}\iint (\partial_{s}w)^2\frac{\w}{1-y^2}\y \t,
\end{equation}
where  $w=w_{x_0,T_0}$ is defined in \eqref{scaling}.
\end{thm}

\medskip

\begin{nb}
We have chosen to present our main  result as Theorem \ref{t1} since
  the existence of a Lyapunov functional in similarity variables is far from being trivial   and it  represents the crucial  step
  in this paper.
\end{nb}
\begin{nb}
Since we crucially need a covering technique in our argument, in fact, we need a uniform version for $x$ near $x_0$
 (see Theorem \ref{t1}'  below).
\end{nb}
\begin{nb}
Let us note that 
our method breaks down in the  case of a   characteristic point, since in the construction of the 
Lyaponov functional 
in similarity variables,  we use  a covering technique in our argument which is not available  at a  characteristic point.
At this moment,  we do not know whether Theorem \ref{t1} continues to hold if  $x_0$ is a   characteristic point.
\end{nb}

As we said earlier, the existence of this Lyapunov functional  $ L(w(s),s)$ together with a blow-up criterion for
equation \eqref{A} make  a crucial step in the derivation of the blow-up
rate for equation \eqref{gen}. Indeed, with the functional $ L(w(s),s)$ and some more work, we are able to
adapt the analysis performed in  \cite{MZajm03, MZimrn05,MZma05} 
 for equation \eqref{NLW} and obtain the following result:

\begin{thm}\label{t2}
{\bf {(Blow-up rate for equation \eqref{gen})}}.\\
Consider   $u $    a solution of ({\ref{gen}})  in one space dimension ($N=1$), with blow-up graph
$\Gamma:\{x\mapsto T(x)\}$ and  $x_0$  a non characteristic point.
Then there exist  $\widehat{S}_2$  large enough  such that

i)
 For all
 $s\ge \widehat{s}_2(x_0)=\max(\widehat{S}_2,-\log \frac{T(x_0)}4)$,
\begin{equation*}
0<\varepsilon_0\le \|w_{x_0}(s)\|_{H^{1}((-1,1))}+ \|\partial_s
w_{x_0}(s)\|_{L^{2}((-1,1))} \le K,
\end{equation*}
where $w_{x_0}=w_{x_0,T(x_0)}$ is defined in (\ref{scaling}).\\
ii)  For all
  $t\in [t_2(x_0),T(x_0))$, where  $t_2(x_0)=T(x_0)-e^{-\widehat{s}_2(x_0)}$, we have
\begin{align}
&&0<\varepsilon_0\le \frac1{\psi_{T(x_0)}(t)}\frac{\|u(t)\|_{L^2(I(x_0,{T(x_0)-t}))}}{ \sqrt{T(x_0)-t}}\label{main1}\\
&&+ \frac{T(x_0)-t}{\psi_{T(x_0)}(t)}\Big
(\frac{\|\partial_tu(t)\|_{L^2(I(x_0,{T(x_0)-t}))}}{
\sqrt{T(x_0)-t}}+
 \frac{\|\partial_x u(t)\|_{L^2(I(x_0,{T(x_0)-t}))}}{ \sqrt{T(x_0)-t}}\Big )\le K,\nonumber
\end{align}
where 
$K=K(p,  a, T(x_0), t_2(x_0),\|(u(t_2(x_0)),\partial_tu(t_2(x_0)))\|_{
H^{1}\times
L^{2}(I(x_0,\frac{T(x_0)-t_2(x_0)}{\delta_0(x_0)}) )})$, \\  $\psi_{T(x_0)}(t)$ is defined in \eqref{psi}, $I(x_0,t)=(x_0+t,x_0-t)$ and $\delta_0(x_0)$ is defined in  \eqref{nonchar}.
\end{thm}

\begin{nb}
As in the pure power nonlinearity  case  \eqref{NLW}, the proof of Theorem \ref{t2} relies on four ideas (the existence of a Lyapunov functional,  interpolation in Sobolev spaces, some
critical Gagliardo-Nirenberg estimates and
a covering technique adapted to the geometric shape of the blow-up surface). It happens that adapting the proof of
\cite{MZimrn05} given in the  pure power nonlinearity  case \eqref{NLW} is straightforward. Therefore, we only present the key argument dedicated to the control of the 4th term in \eqref{perb4}, and refer to \cite{MZajm03, MZimrn05, MZma05}  for the treatment 
of the terms
appearing in the definition of $E_1(w(s),s)$ defined in \eqref{5jan1} and refer to 
 \cite{HZjhde12, HZnonl12,H1, omar1, omar2} 
 for the  control of the  three first  terms of \eqref{perb4}
  for the rest of the proof.
\end{nb}

\begin{nb}
Since we crucially need a covering technique in the  argument of the construction of the Lyapunov functional,
our method breaks down in the case  of a characteristic point  and we are not able to obtain the sharp estimate
as in the unperturbed case \eqref{NLW}. 
\end{nb}

\begin{nb}\label{1.7}
It should be noted here that the restriction  to a one dimensional space is due to the use of  the embedding  $H^1(\R)\hookrightarrow 
L^{\infty}(\R)  $. Unfortunately,  as we pointed   in  the construction of the Lyapunov  functional,
our method breaks down in the case  of higher dimensions,   and we are not able to obtain the sharp estimate
  as in the case of  pure power nonlinearity   \eqref{NLW}. However, as already stated in Theorem \ref{tp1} above,   we can 
 derive  a 
 polynomial in $s$ space-time estimate    in higher dimension  in the subconformal  case ($1<p< \frac{N+3}{N-1}$).
\end{nb}

\begin{nb}
Let us remark  we can  obtain the same  blow-up rate for  the  more general equation 
\begin{equation}\label{NLWlog}
\partial_t^2 u =\partial_x^2 u+|u|^{p-1}u\log^a(2+u^2)+k(u),\,\,\,(x,t)\in \R \times [0,T),
 \end{equation}
under the assumption  that $|k(u)|\leq M(1+|u|^{p}\log^b(2+u^2))$,   for some $M > 0$ and $b<a-1$. 
More precisely,   under this hypothesis,   we can construct  a suitable Lyapunov
functional for  this equation. Then, we  can prove a similar result to $\eqref{main1}$. However, the case where $a-1\le b<a$ seems 
to be out reach with our technics, though we think  we may obtain the same rate as in the unperturbed  case.
\end{nb}


%
%

\medskip

This paper is organized as follows: In Section \ref{section2}, we obtain a rough control 
of the solution $w$ in the subconformal  case. In Section \ref{section3}, in one space dimension and  thanks to the  result obtained, we prove that the 
functional $L(w(s),s)$  is a Lyapunov functional for equation \eqref{A}. Thus, we get Theorem \ref{t1}. Finally, 
applying this last theorem,  we prove
Theorem \ref{t2}.

\section{A polynomial bound for  solution of equation  \eqref{A}} \label{section2}
Consider $u$ a solution of \eqref{gen}  with blow-up graph  $\Gamma:\{x\mapsto T(x)\}$ and  $x_0$  a 
non characteristic point.
This section is devoted to deriving   a uniform version  of  Theorem \ref{tp1} valid  for $x$ near $x_0$.
  More precisely, this is the aim of this section.

\noindent {\bf{Theorem} \ref{tp1}'} 
\label{tp1bis}
{\it
\noindent  Consider $u $   a solution of ({\ref{gen}}) with
blow-up graph $\Gamma:\{x\mapsto T(x)\}$ and  $x_0$  a non
characteristic point. Then,  there exists
$t_{0}(x_{0})\in [ 0,T(x_{0}))$ and  $q=q(a,p,N)>0$  such that,   for all $T_0 \in (t_{0}(x_0),T(x_{0})]$,  for all $s\geq -\log (T_0-t_{0}(x_{0}))$ and $x\in \er^N$ where $|x-x_0|\le \frac{e^{-s}}{\delta_0(x_0)}$, 
we have
\begin{equation}\label{feb19}
\int_{s}^{s+1}\!\ibint \big( w^2(y,\tau )+|\grad w(y,\tau )|^2 +( \partial_{s}w(y,\tau ))^2\big)\y \t \leq K_1 s^q,
\end{equation}
where $w=w_{x,T^*(x)}$ is defined in \eqref{scaling}, with
 \begin{equation}\label{18dec1}
T^*(x)=T_0-\delta_0(x_0)(x-x_0)
\end{equation}
 and $\delta_{0}(x_{0})$ defined in \eqref{nonchar}.  Note that  $K_1$ depends on $ p, a, N,  \delta_{0}(x_{0})$, 
 $T(x_{0})$, $t_0(x_0)$ and
$\|(u(t_0(x_0)),\partial_tu(t_0(x_0)))\|_{H^{1}\times
L^{2}(B(x_0,\frac{T(x_0)-t_0(x_0)}{\delta_0(x_0)}) )}$.
}

%

\bigskip

In order to prove this theorem, we need to construct a  Lyapunov functional for equation \eqref{A}. 
In order to do so,  we start by introducing the following 
functionals:
\begin{align}\label{F0}
E_{N}(w(s),s)=&\ibint \Big(\frac{1}{2}(\partial_{s}w)^2+\frac{1}{2}|\grad w|^2-\frac12|y.\grad w|^2\nonumber\\
&+\frac{p+1}{(p-1)^2}w^2-e^{-\frac{2(p+1)s}{p-1}}s^{\frac{2a}{p-1}}   F(\p w)\Big)\w \y, \nonumber\\
J_{N}(w(s),s)=&-\frac1{s} \ibint   w\partial_{s}w\w \y, \\
H_{N,m}(w(s),s)=&E_{N}(w(s),s)+m J_{N}(w(s),s),\nonumber
\end{align}
where $F$ is given by \eqref{defF}  and $m >0$ is a sufficiently large constant that will be fixed later.

\medskip

As we see above, the target of this section is to prove, for some $\m_0$ large enough,
that the energy $H_{\m_0,N}$ satisfies  the  following inequality: 
\begin{equation}\label{Heta0}
\frac{d}{ds}H_{\m_0,N}(w(s),s))\le - \alpha \ibint (\partial_{s}w)^2\frac{\w}{1-|y|^2}\y
+\frac{\m_0 (p+3)}{2s}H_{\m_0,N}(w(s),s)
+  C e^{-2s},
 \end{equation}
which implies  that 
$H_{m _0,N} (w(s),s)$ satisfies the following polynomial estimate:
\begin{equation}\label{Heta1}
H_{\m_0,N}(w(s),s))\le  K s^{\mu_0},
 \end{equation}
for some $K>0$ and  $\mu_0>0$.
%
%
\medskip

In the remaining part of this section, we
consider $u$ a solution of \eqref{gen}  with blow-up graph  $\Gamma:\{x\mapsto T(x)\}$ and  $x_0$  a 
non characteristic point.
Let   $T_0\in (0, T(x_0)]$, for all 
$x\in \er^N$  such that $|x-x_0|\le \frac{T_0}{\delta_0(x_0)}$, where 
$\delta_0\in(0,1)$  is defined in \eqref{nonchar} and  we write $w$ instead of $w_{x,T^*(x)}$ defined in (\ref{scaling}) 
with $T^*(x)$ given by \eqref{18dec1}.

\subsection{Classical energy estimates}
In this subsection, we state two lemmas which are crucial for the construction of a  Lyapunov functional.
 We begin with bounding
the time derivative of   $E_{N}(w(s),s)$  in the following lemma:
\begin{lem}\label{lem22}  For all   
  $s \geq \max(-\log T^*(x), 1)$, we have 
\begin{align}\label{E011}
\frac{d}{ds}E_{N}(w(s),s)\le& -  \frac{3\alpha}{2}\ibint (\partial_{s}w)^2\frac{\w}{1-|y|^2}\y\\
&+  \frac{C}{s^{a+1}}\ibint |w|^{p+1}\log^a(2+\p^2w^2)\w \y+\Sigma_{1}(s),\no
 \end{align}
where $\Sigma_{1}(s)$  satisfies
\begin{align}\label{E011bis}
\Sigma_{1}(s)\le \frac{C}{s^2}\ibint  |\grad w|^2(1-|y|^2)\w \y+\frac{C}{s^2}\ibint  w^2\w \y  +C e^{-2s}.
 \end{align}
\end{lem}
{\it Proof}: Multiplying $\eqref{A}$ by $\partial_{s} w\!\ \w$ and integrating over  $B$, we obtain
\begin{align}
\frac{d}{ds}E_{N}(w(s),s)=& - 2\alpha \ibint (\partial_{s}w)^2\frac{\w}{1-|y|^2}\y\label{E00}\\
&+\underbrace{\frac{2p+2}{p-1}
e^{-\frac{2(p+1)s}{p-1}}s^{\frac{2a}{p-1}}\ibint\big( F(\p w)-\frac{\p wf(\p w)}{p+1}\big)\w \y}_{\Sigma^1_{1}(s)}\no\\
&\underbrace{-\frac{2a}{p-1}e^{-\frac{2(p+1)s}{p-1}}s^{\frac{2a}{p-1}-1}\ibint \big( F(\p w)-\frac{\p wf(\p w)}{2}\big)\w \y}_{\Sigma^2_{1}(s)}\no\\
&+\underbrace{\gamma (s)\ibint w\partial_sw\w\y
+\frac{2a}{(p-1)s}\ibint (\partial_{s}w)^2\w\y}_{\Sigma^3_{1}(s)}\no\\
&+\underbrace{
\frac{2a}{(p-1)s}\ibint  y.\grad w\partial_{s}w\w \y
}_{\Sigma^4_{1}(s)}.\no
 \end{align}
Now, we control the terms $\Sigma_{1}^{1}(s)$, $\Sigma_{1}^{2}(s)$, $\Sigma_{1}^{3}(s)$  and  $ \Sigma_{1}^{4}(s)$. 
Note from \eqref{defF123},   \eqref{equiv2}  and  \eqref{equiv3}  that
\begin{equation}\label{id2}
 F(\p w)-\frac{\p wf(\p w)}{p+1}\le   C+C \frac{ \p w}{s}f(\p w),
\end{equation}
which implies,  for all $s\ge \max  (-\log T^*(x),1)$,
\begin{equation}\label{sigma11}
\Sigma_{1}^{1}(s)\le C
e^{-\frac{2(p+1)s}{p-1}}s^{\frac{2a}{p-1}-1}
\ibint  \p wf(\p w) 
\w \y+  C e^{-2s}.
\end{equation}
Let us recall, from the expression of $\p= \ps$ defined in \eqref{defphi},  that  we have, for all $s\ge \max  (-\log T^*(x),1)$, 
\begin{equation}\label{id1}
e^{-\frac{2(p+1)s}{p-1}}s^{\frac{2a}{p-1}}  \p wf(\p w)=\frac1{s^{a}}|w|^{p+1}\log^a(2+\p^2w^2).  
\end{equation}
Thus, using \eqref{sigma11} and \eqref{id1}, we obtain, for all $s\ge \max  (-\log T^*(x),1)$,
\begin{equation}\label{sigma111}
\Sigma_{1}^{1}(s)\le
\frac{C}{s^{a+1}}\ibint |w|^{p+1}\log^a(2+\p^2w^2)\w \y+  C e^{-2s}.
\end{equation}
Similarly, by
 \eqref{equiv1} and \eqref{id1},  we obtain easily, for all $s\ge \max  (-\log T^*(x),1)$,
\begin{equation}\label{sigma12}
\Sigma_{1}^{2}(s)\le \frac{C}{s^{a+1}}\ibint |w|^{p+1}\log^a(2+\p^2 w^2)\w \y+  C e^{-2s}.
\end{equation}
%
  By using  the following  basic inequality
\begin{equation}\label{basic1}
ab\le \varepsilon a^2+\frac1{\varepsilon}b^2, \ \forall \varepsilon >0,
\end{equation}
and the expression of $\gamma(s)$ defined in 
  \eqref{defgamma},  we  write, for all $s\geq \max  (-\log T^*(x),1)$ 
\begin{equation}\label{sigma13}
\Sigma_{1}^{3}(s)+
\Sigma_{1}^{4}(s)\leq \frac1{p-1} \ibint (\partial_{s}w)^2\frac{\w}{1-|y|^2}\y+ \frac{C}{s^2}\ibint  \Big(|\grad w|^2(1-|y|^2)+w^2\Big)\w \y.
\end{equation}
The result \eqref{E011} and \eqref{E011bis} follows  immediately from  \eqref{E00},  \eqref{sigma111}, \eqref{sigma12} and  \eqref{sigma13},
which ends the proof of Lemma \ref{lem22}.
\Box
\medskip
\begin{nb}
By showing the estimate proved in Lemma \ref{lem22}, related to the so called  natural functional $E_{N}(w(s),s)$, we have  some
nonnegative  terms  in the right-hand side of \eqref{E011} and this  does not allow to construct a decreasing
 functional (unlike the case of a  pure power nonlinearity). 
The main problem   is related to  the nonlinear term 
$$\frac{1}{s^{a+1}}\ibint |w|^{p+1}\log^a(2+\p^2(s)w^2)\w \y=\frac{1}{s}\ibint w
e^{-\frac{2ps}{p-1}}s^{\frac{a}{p-1}} f(\ps w)\w \y.$$     
 To overcome this problem, we adapt the strategy  used in  \cite{HZjhde12, HZnonl12,H1, omar1, omar2}. More precisely, by using the identity obtained 
 by multiplying  equation \eqref{gen} by $w\w$, 
then integrating over $B$, we can 
  introduce a  new functional $H_{m,N}$,  defined in \eqref{F0} where $m>0$ is sufficiently large and
 will be fixed such that $H_{m,N}$ satisfies a differential inequality similar to  \eqref{jardin1}.  
\end{nb}

 \medskip

We are going to prove the following estimate on the functional $J_{N}(w(s),s)$. 
 \begin{lem}\label{LemJ_0}
For  all $s \geq \max  (-\log T^*(x),1)$, we have 
\begin{align}\label{6nov2018}
\frac{d}{ds}J_{N}(w(s),s)\ \  \le&\ \ \frac{p+3}{2s}E_{0,N}(w(s),s)-\frac{p+7}{4s}\ibint   (\partial_{s}w)^2\w \y\\
&-\frac{p-1}{4s}\ibint ( |\grad w|^2-(y.\grad w)^2)\w\y
- \frac{p+1}{2(p-1)s} \ibint  w^2\w\y\no\\
&
-\frac{p-1}{2(p+1)s^{a+1}}\ibint |w|^{p+1}\log^a(2+\p^2 w^2)  
\w\y+\Sigma_2(s),\no
\end{align}
where  $\Sigma_2(s)$ satisfies
  \begin{align}\label{DEC1}
  \Sigma_{2}(s)\leq& \frac{C}{\sqrt{s}}
 \ibint (\partial_{s}w)^2 \frac{\w}{1-|y|^2}\y+\frac{C}{s\sqrt{s}}\iint   |\grad w|^2(1-|y|^2)\w\y \\
&+\frac{C}{s\sqrt{s}}\ibint  w^2\w\y+\frac{C}{s^{a+2}}\ibint |w|^{p+1}\log^a(2+\p^2w^2)\w \y+  C e^{-2s}.\no
\end{align}
\end{lem}
{\it Proof}: Note that $J_{N}(w(s),s)$ is a differentiable function and   that we get  for all 
 $s \geq \max  (-\log T^*(x),1)$, 
\begin{equation*}\label{j1}
\frac{d}{ds}J_{N}(w(s),s)=-\frac1{s} \ibint   
(\partial_{s}w)^2\w \y-\frac1{s} \ibint   w\partial^2_{s}w\w\y +\frac{1}{s^2} \ibint   w\partial_{s}w\w \y.
\end{equation*}
From equation  $\eqref{A}$, we obtain
\begin{align}\label{j2}
\frac{d}{ds}J_{N}(w(s),s)=&\J \ibint ((\grad w)^2-|y.\grad w|^2)\w \y
-\frac2{s}  \ibint  \partial_swy.\grad w\w\y\no\\
&-\frac1{s^{a+1}}\iint |w|^{p+1}\log^a(2+\p^2w^2)\w \y-\frac{2a}{(p-1)s^2} \ibint   
wy.\grad w\w \y
\no \\
&+\J \Big(\frac{p+3}{p-1}-2N-\frac{2a+1-p}{(p-1)s}\Big) \ibint w\partial_s w\w \y-\J \ibint   
(\partial_{s}w)^2\w \y 
\no\\
&+\J \Big( \frac{2p+2}{(p-1)^2}-\gamma (s) \Big)
  \ibint  w^2\w\y
+\frac{4\alpha}{s} \ibint  w\partial_s w \frac{|y|^2 \w}{1-|y|^2} \y.\no
 \end{align}
According to the expressions of $E_0(w(s),s)$,  $\ps$ defined in \eqref{F0} and  \eqref{defphi} and the identity \eqref{id1} 
with some straightforward computation,  we obtain  \eqref{6nov2018} where
 \begin{equation}\label{R0}
\Sigma_2(s)=\Sigma^1_2(s)+\Sigma_{2}^{2}(s),
\end{equation} 
and 
\begin{align*}
\Sigma_{2}^1(s)=&
\frac{p+3}2e^{-\frac{2(p+1)s}{p-1}}s^{\frac{2a}{p-1}-1} 
  \ibint \Big( F(\p w)-\frac{\p wf(\p w)}{p+1}\Big)\w\y,  \no \\
\Sigma_{2}^2(s)=&-\frac2{s}  \ibint  \partial_swy.\grad w\w\y-\frac{\gamma (s)}{s} 
  \ibint  w^2\w\y\no\\
&+\J \Big(\frac{p+3}{p-1}-2N+\frac{p-1-2a}{(p-1)s}\Big) \ibint w\partial_s w\w \y
\no\\
&+\frac{4\alpha}{s} \ibint  w\partial_s w \frac{|y|^2 \w}{1-|y|^2} \y-\frac{2a}{(p-1)s^2} \ibint   
wy.\grad w\w \y.
\no 
 \end{align*}

\medskip

We are going now to estimate  the  different terms of \eqref{R0}.
Thanks to \eqref{id1} and  \eqref{id2},  we conclude that  for all $s\ge \max  (-\log T^*(x),1)$
\begin{equation}\label{R1}
\Sigma_{2}^{1}(s)\le
\frac{C}{s^{a+2}}\ibint |w|^{p+1}\log^a(2+\p^2w^2)\w \y+  C e^{-2s}.
\end{equation}
By using the inequality \eqref{basic1} and \eqref{defgamma},    we conclude that 
  for all $s\ge \max  (-\log T^*(x),1),$
  \begin{align}\label{R2}
  \Sigma_{2}^{2}(s)\leq& \frac{C}{\sqrt{s}}
 \ibint (\partial_{s}w)^2 \frac{\w}{1-|y|^2}\y+\frac{C}{s\sqrt{s}}\ibint  |\grad w|^2(1-|y|^2)\w\y\no \\
&+\frac{C}{s\sqrt{s}}\ibint  w^2\frac{\w}{1-|y|^2}\y.
  \end{align}
Let us recall  from \cite{MZajm03} the following Hardy type inequality
\begin{equation}\label{0Hardybis}
 \ibint w^2\frac{|y|^2\w}{1-|y|^2}\y\leq C\ibint  |\grad w|^2(1-|y|^2) \w \y+C\ibint w^2\w \y.
\end{equation}
 (see the appendix  in \cite{MZajm03} for a proof). 
Using \eqref{0Hardybis} and  the fact that $\frac{\w}{1-|y|^2}=\w+\frac{|y|^2\w}{1-|y|^2}$, we get 
\begin{equation}\label{Hardybis}
 \ibint w^2\frac{\w}{1-|y|^2}\y\leq C\ibint  |\grad w|^2(1-|y|^2) \w \y+C\ibint w^2\w \y.
\end{equation}
 Thus, it follows from $\eqref{R2}$ and  $\eqref{Hardybis}$ that
for all $s\ge \max(-\log T^*(x),1)$,
  \begin{align}\label{R3}
  \Sigma_{2}^{2}(s)\leq \frac{C}{\sqrt{s}}
 \ibint (\partial_{s}w)^2 \frac{\w}{1-|y|^2}\y+\frac{C}{s\sqrt{s}}\ibint   |\grad w|^2(1-|y|^2)\w\y+\frac{C}{s\sqrt{s}}\ibint  w^2\w\y.
  \end{align}
Consequently, collecting \eqref{R0},   \eqref{R1} and  \eqref{R3}, one easily obtains that $\Sigma_{2}(s)$ satisfies \eqref{DEC1},
 which ends the proof of Lemma \ref{LemJ_0}.
%
%
%
%
 \Box

\subsection{Existence of a decreasing functional for equation \eqref{A}}
In this subsection, by using Lemmas  \ref{lem22} and \ref{LemJ_0}, we are going to construct a decreasing
functional for equation (\ref{A}). Let us define the following functional:
\begin{equation}\label{lyap1}
N_{m,N}(w(s),s)=s^{-\frac{m(p+3)}{2}}H_{m,N}(w(s),s)+ \sigma (m) e^{-s},
\end{equation}
 where $H_{m,N}$ is defined in \eqref{F0},  and   $m$ and  $\sigma=\sigma(m)$ are   constants that  will be determined later.

\medskip

We now state the following proposition:

\begin{prop}\label{proplyap}
There exist $m_0>1$, $\sigma_0>0$,  $S_{1}\geq 1$ and $\lambda_1>0$, such that for all $s \geq \max  (-\log T^*(x),S_{1})$, 
we have the following inequality:
\begin{align}\label{DEC101} 
N_{m_0,N}(w(s+1),s+1)-N_{m_0,N}(w(s),s)\le& - \frac{2}{(p-1)s^b}\ia\ibint\!\!
\frac{(\partial_{s}w)^2}{1-|y|^2}\w
\y\t\\
&\!\!\!-\frac{\lambda_1}{s^{b+1}} \ia\ibint ( |\grad w|^2-(y.\grad w)^2)\w\y\t\no\\
& -\frac{\lambda_1}{s^{a+b+1}} 
\ia \ibint |w|^{p+1}\log^a(2+\p^2w^2)  
\w\y\t\no\\
&-\frac{\lambda_1}{s^{b+1}}\ia\ibint\!\!
w^2\w
\y\t,\qquad\no
\end{align}
where
\begin{equation}\label{defb}
b=\frac{m_0(p+3)}{2}.
\end{equation}
Moreover, there exists
 $S_{2}\geq S_1$ such that for all $s \geq \max  (-\log T^*(x),S_{2})$, 
we have 
 \begin{equation}\label{posi1}
N_{m_0,N}(w(s),s)\geq 0.
\end{equation}
\end{prop}

\vspace{0.3cm}

{\it Proof:}
From  the definition of $H_m$ given  in \eqref{F0},   Lemmas  \ref{lem22}, \ref{LemJ_0} and the classical inequality 
$J_N(w(s),s)\le  \frac1{\sqrt{s}}\ibint (\partial_{s}w)^2\w\y+ \frac1{s\sqrt{s}}\ibint w^2\w\y,$ we can write for all $s\ge \max(-\log T^*(x), 1)$,
\begin{align}\label{DEC9}
\frac{d}{ds}H_{m,N}(w(s),s)\le &
-\Big(\frac{m(p-1)}{2(p+1)}-C_0-\frac{C_0 m}{s}\Big)
\frac{1}{s^{a+1}}
\ibint |w|^{p+1}\log^a(2+\p^2w^2)  
\w\y\no\\
&- \Big(\frac{3}{p-1}-\frac{C_0 m}{\sqrt{s}}\Big)\ibint (\partial_{s}w)^2\frac{\w}{1-|y|^2}\y+\frac{m(p+3)}{2s}H_{m,N}(w(s),s)\no\\
&-\Big(\frac{m(p-1)}{4s}- \frac{C_0m}{s\sqrt{s}}- \frac{C_0}{s^2}\Big)\ibint  ( |\grad w|^2-|y.\grad w|^2)\w \y  \no \\
  &-m\Big(\frac{p+7}{4s}-\frac{C_0}{s\sqrt{s}}\Big)\ibint   (\partial_{s}w)^2\w \y\\
&-\Big( \frac{m(p+1)}{2(p-1)s}-\frac{C_0m}{s\sqrt{s}}-\frac{C_0m}{s^2\sqrt{s}}-\frac{C_0}{s^2}
\Big) \ibint  w^2\w\y+(C_0 m+C_0) e^{-2s},\no
\end{align}
where $C_0$  stands for some universal constant depending only on $N,$ $p$ and $a$.
We first  choose    $m_0$ such that $\frac{m_0(p-1)}{4(p+1)}-C_0=0$,  so 
$$\frac{m_0(p-1)}{2(p+1)}-C_0-\frac{C_0m_0}{s}=m_0\Big(\frac{p-1}{4(p+1)}-\frac{C_0}{s}\Big).$$
We now choose    $S_1=S_1(m_0,a,p,N)$ large enough ($S_1\ge 1$), so that for all $s\ge S_1$, we have
\begin{eqnarray}
\frac{m_0(p-1)}{8(p+1)}-\frac{C_0}{s} \ge 0, \qquad
 \frac{1}{p-1}-\frac{C_0m_0}{\sqrt{s}}\ge 0,\qquad 
\frac{m_0(p-1)}{8}- \frac{C_0m_0}{\sqrt{s}}- \frac{C_0}{s}\ge 0,\no \\
  \frac{p+7}{8}-\frac{C_0}{\sqrt{s}}\ge 0,\qquad \qquad\frac{m_0(p+1)}{4(p-1)}-\frac{C_0m_0}{\sqrt{s}}-\frac{C_0m_0}{s\sqrt{s}}-\frac{C_0}{s}\ge 0.
\end{eqnarray}
Then, we deduce that  for all $s\ge \max(-\log T^*(x), S_1)$,
\begin{align}\label{DEC911}
\frac{d}{ds}H_{m_0,N}(w(s),s)\le &
- \frac{2}{p-1}\ibint (\partial_{s}w)^2\frac{\w}{1-|y|^2}\y+\frac{m_0(p+3)}{2s}H_{m_0,N}(w(s),s)\no\\
&-\frac{\lambda_0}{s^{a+1}}
\ibint |w|^{p+1}\log^a(2+\p^2(s)w^2)  
\w\y\no\\
&-\frac{\lambda_0}{s}\ibint   (|\grad w|^2 -(y.\grad w)^2)\w \y  \\
&-\frac{\lambda_0}{s} \ibint  w^2\w\y+C_0(m_0+1) e^{-2s},\no
\end{align}
 where $\lambda_0=\inf( \frac{m_0(p-1)}{8(p+1)},\frac{m_0(p-1)}{8},
\frac{m_0(p+1)}{4(p-1)},\frac{p+7}8 )$.

By using the definition of $N_{m_0,N}$ given  in  \eqref{lyap1} together with  the estimate \eqref{DEC911}, we  easily prove that  $N_{m_0,N}$ satisfies for all  ${s\ge \max(-\log T^*(x),S_1),}$
\begin{align}\label{M11}
\frac{d}{ds}N_{m_0,N}(w(s),s)\le &
- \frac{2}{(p-1)s^b}\ibint (\partial_{s}w)^2\frac{\w}{1-|y|^2}\y\no\\
&-\frac{\lambda_0}{s^{a+b+1}}
\ibint |w|^{p+1}\log^a(2+\p^2(s)w^2)  
\w\y\no\\
&-\frac{\lambda_0}{s^{b+1}}\ibint   (|\grad w|^2-(y.\grad w)^2)\w \y  
\\
&-\frac{\lambda_0}{s^{b+1}} \ibint  w^2\w\y-e^{-s}\Big(\sigma -C_0(m_0+1)\frac{e^{-s}}{s^b}\Big).\no
\end{align}
We now choose $\sigma =C_0(m_0+1)e^{-S_1}$,  so we have, for all  $s\ge S_1$
 \begin{equation}\label{11jan1}
\sigma -C_0(m_0+1)\frac{e^{-s}}{s^b}\ge 0.  
\end{equation}
By  integrating  in time between $s$ and $s+1$
the  inequality
(\ref{M11}) and using \eqref{11jan1}, we easily obtain
(\ref{DEC101}).
This concludes the proof of the first part of Proposition \ref{proplyap}.

\medskip

 We  prove    \eqref{posi1} here. 
  The argument is the
same as in the corresponding part in \cite{
HZjhde12, HZnonl12,H1, omar1, omar2}. We write the
proof for completeness. Arguing by contradiction, we assume that
there exists   
 $s_1 \geq \max  (-\log T^*(x),S_{2})$ such that $N_{m_0,N}(w(s_1),s_1)<0$, where $S_2=S_2(a,p,N)$  is large enough, 
$w=w_{x,T^*(x)}$. Since the energy $N_{m_0,N}(w(s),s)$ decreases in time, we
have $N_{m_0,N}(w(1+s_1),1+s_1)<0$.

 Consider now for $\delta>0$ the function
$\widetilde{w}^{\delta}(y,s)=w_{x,T^{*}(x)-\delta }(y,s)$. From
(\ref{scaling}), we see that for all $(y,s)\in B\times
[1+s_1,+\infty)$
\begin{equation}\label{blowupbefore}
 \widetilde{w}^{\delta}(y,s)=\frac{\p (-\log (\delta+e^{-s}))}{\ps}
w( \frac{y}{1+\delta e^s},-\log (\delta+e^{-s})),
\end{equation}
where $\p$ defined in \eqref{defphi}. Then, we make the following 3 observations: 
\begin{itemize}
\item  (A) Note that $\widetilde{w}^{\delta}$ is defined in  $ B\times [1+s_1,+\infty)$,
whenever $\delta>0$ is small enough such that
$-\log(\delta+e^{-1-s_1})\ge s_1.$
\item (B) By construction, $\widetilde{w}^{\delta}$
is also a solution of equation (\ref{A}).
\item (C) For $\delta$ small enough, we have
$N_{m_0,N}(\widetilde{w}^{\delta}(1+s_1),1+s_1)<0$ by continuity of the function
$\delta \mapsto N_{m_0,N}(\widetilde{w}^{\delta}(1+s_1),1+s_1)$.
\end{itemize}
Now, we fix $\delta=\delta_1>0$ such that (A), (B) and (C) hold. Since
 $N_{m_0,N}(\widetilde{w}^{\delta_1}(s),s)$ is decreasing  in time, we have
 \begin{equation}
\label{255}
\liminf_{s\rightarrow +\infty}N_{m_0,N}(\widetilde{w}^{\delta_1}(s),s)\le N_{m_0,N}(\widetilde{w}^{\delta_1}(1+s_1),1+s_1)<0,
 \end{equation}
on the one hand.  On the other hand, from  \eqref{basic1}, we have 
\begin{equation}\label{c55}
-\frac{m_0}{s} \int_{B}\widetilde{w}^{\delta_1}\partial_s\widetilde{w}^{\delta_1}\w
{\mathrm{d}}y\ge-\frac{m_0}{s}
\ibint(\partial_s\widetilde{w}^{\delta_1})^2\w {\mathrm{d}}y-\frac{m_0}{s}
\int_{B}(\widetilde{w}^{\delta_1})^2\w {\mathrm{d}}y.
\end{equation}
By \eqref{F0},  \eqref{c55} and  for sufficiently large $ S_2$,  we deduce that 
\begin{equation}
H_{m_0,N}(\widetilde{w}^{\delta_1}(s),s)\ge 
-e^{-\frac{2(p+1)s}{p-1}}s^{\frac{2a}{p-1}}  \ibint  F(\p \widetilde{w}^{\delta_1})\w \y, \quad \forall s\ge \max (s_1+1,S_2). \nonumber
\end{equation}
So, by \eqref{lyap1}, we have
\begin{equation}
N_{m_0,N}(\widetilde{w}^{\delta_1}(s),s)\ge 
-e^{-\frac{2(p+1)s}{p-1}}s^{\frac{2a}{p-1}-b}  \ibint  F(\p \widetilde{w}^{\delta_1})\w \y
. \nonumber
\end{equation}
Due to  \eqref{equiv4}, we infer,
\begin{align}\label{N0}
N_{m_0,N}(\widetilde{w}^{\delta_1}(s),s)\ge
-C e^{-\frac{2(p+1)s}{p-1}}s^{\frac{2a}{p-1}-b}  \ibint  |\p \widetilde{w}^{\delta_1}|^{\bar p+1}  \y
-Ce^{-2s}.
\end{align}
 Notice that,
after a change of variables defined in \eqref{blowupbefore}, we find that
\begin{equation*}
  \ibint  |\p \widetilde{w}^{\delta_1}|^{\bar p+1 } \y
=(1+\delta_1e^{s})^N\p^{\bar p+1}\big(-\log (\delta_1+e^{-s})\big)  \ibint 
|w(z,-\log
(\delta_1+e^{-s}))|^{\bar p+1} {\mathrm{d}}z.\
\end{equation*}
Since we have $-\log (\delta_1+e^{-s})\rightarrow -\log \delta_1$
as $s\rightarrow +\infty$, then
$\p(-\log (\delta_1+e^{-s})\rightarrow \p(-\log (\delta_1))$. Moreover, by exploiting \eqref{subc} and  \eqref{defpb}, we have 
$\bar p<1+\frac4{N-2}$.  Then    $\|w(s)\|_{L^{\bar p+1}(B)}$ is
locally bounded, by a continuity argument, it follows that the
former integral remains bounded and
\begin{align}\label{N0}
N_{m_0,N}(\widetilde{w}^{\delta_1}(s),s)\ge
-C{(\delta_1+e^{-s})^N}e^{-(1+2\alpha)s}s^{\frac{2a}{p-1}-b}
-Ce^{-2s}
\rightarrow 0.
\end{align}
as $s\rightarrow +\infty$. So, it follows that
\begin{equation}\label{d2}
\liminf_{s\rightarrow +\infty}N_{m_0,N}(\widetilde{w}^{\delta_1}(s),s)\ge 0 .
\end{equation}

\noindent
  From \eqref{255}, this is a contradiction. Thus \eqref{posi1} holds.
 This concludes  the  proof of Proposition \ref{proplyap}.
\Box

\subsection{Proof of Theorem \ref{tp1}'}
 We define the following time:
\begin{equation}\label{new19dec1}
 t_0(x_0)=\max(T(x_0)-e^{-S_2},0).
\end{equation}
According to the Proposition \ref{proplyap}, we obtain the following corollary which summarizes the principle properties of $N_{m_0,N}(w(s),s)$ defined in   \eqref{lyap1}.
\begin{cor}\label{19dec3} {\bf (Estimate on $N_{m_0,N}(w(s),s)$).}
  There exists  $t_0(x_0)\in [0, T(x_0))$
such that, for all $T_0\in (t_0(x_0),T(x_0)]$,  for all $s\ge  -\log (T_0-t_0(x_0))$
 and $x\in \er^N$ where $|x-x_0|\le \frac{e^{-s}}{\delta_0(x_0)}$,  
 we have
\begin{equation}\label{2018N0}
0\leq N_{m_0,N}(w(s),s) \leq N_{m_0,N}(w(\tilde{s}_{0}),\tilde{s}_{0}) ,
\end{equation}
\begin{equation}\label{18jan1}
 \int_{s}^{s+1}\int_{B}\Big(|\grad w(y,\tau)|^2(1-|y|^2)+\frac{(\partial_sw(y,\tau))^{2}}{1-|y|^2}+w^2(y,\tau)\Big) \w  {\mathrm{d}}y{\mathrm{d}}\tau\leq C\Big(1 +N_{m_0,N}(w(\tilde{s}_{0}),\tilde{s}_{0} ) \Big)s^{b+1},
\end{equation}
where $w=w_{x,T^*(x)}$ is defined in \eqref{scaling}, with 
 $T^*(x)$  given in \eqref{18dec1}
where $\tilde{s}_{0}=-\log (T^*(x)-t_0(x_0))$ and $b$ is defined in \eqref{defb}.
\end{cor}

\medskip

\begin{nb} Using the definition of  (\ref{scaling}) of
$w_{x,T^*(x)}=w$, we write easily
\begin{equation}\label{2018N1}
   N_{m_0,N}(w(\widetilde{s_0}), \widetilde{s_0})
\le \widetilde{K_0},
\end{equation}
where
$\widetilde{K_0}=\widetilde{K_0}(T(x_0)-t_0(x_0),\|(u(t_0(x_0)),\partial_tu(t_0(x_0)))\|_{H^{1}\times L^{2}(B(x_0,\frac{T(x_0)-t_0(x_0)}{\delta_0(x_0)}))})$.
\end{nb}
\medskip

 With    Corollary \ref{19dec3}, we are in a position to  prove  Theorem \ref{tp1}'
 which is a uniform
 version of Theorem \ref{tp1} for $x$ near $x_0$.

%
%
%
%
%
%
%

\medskip

{{\it {Proof of Theorem \ref{tp1}'}}}:
 Note that the estimate on the space-time $L^2$ norm of $\partial_sw$ was already
proved in Corollary  \ref{19dec3} (take $q=b+1$, where   $b$ is defined in \eqref{defb}). Thus we focus on the space-time   $L^2$ norm
of  $\grad w$.  Let us remark that this estimate  already follows from  Corollary  \ref{19dec3} with the ball $B$   replaced by  $B(0,\frac12)$.
Thanks to the covering technique (we refer the reader to Merle
 and Zaag \cite{MZimrn05} (pure power case) and Hamza and Zaag in  Lemma 2.8 in \cite{HZjhde12}),  we easily extend this estimate from $B(0,\frac12 )$ to $B$. This concludes the proof of  Theorem \ref{tp1}'.
\Box

\section{Proof of Theorem \ref{t1} and Theorem \ref{t2}}\label{section3}

In this section, we consider the one space dimensional case ($N=1$). We prove 
Theorem \ref{t1} and Theorem \ref{t2} here. Before doing that,  since we consider the one space dimensional case and  thanks to  Theorem \ref{tp1},   we  first   prove  a  polynomial  estimate.
This section is divided into three parts:
\begin{itemize}
\item   In subsection \ref{3.1},    we prove a  polynomial  estimate.
\item  In subsection \ref{3.2},   we state a general version of Theorem \ref{t1}, uniform for $ x$ near $x_0$ and prove it.
%
\item In subsection \ref{3.3}, we prove   Theorem \ref{t2}.
\end{itemize}
\subsection{Polynomial estimate }\label{3.1}
\begin{prop}\label{prop3.1}
\noindent  Consider $u $   a solution of ({\ref{gen}}) with
blow-up graph $\Gamma:\{x\mapsto T(x)\}$ and  $x_0$  a non
characteristic point. Then,  there exists
$t_{0}(x_{0})\in [ 0,T(x_{0}))$ and  $q_1=q_1(a,p,N)>0$  such that,   for all $T_0 \in (t_{0}(x_0),T(x_{0})]$,  for all $s\geq -\log (T_0-t_{0}(x_{0}))$
and $x\in \er$ where $|x-x_0|\le \frac{e^{-s}}{\delta_0(x_0)}$, 
we have
\begin{equation}\label{co2bis}
\|w(s)\|_{H^1((-1,1))}+\|\partial_sw(s)\|_{L^2((-1,1))}
\leq K_2 s^{q_1},
\end{equation}
where $w=w_{x,T^*(x)}$ is defined in \eqref{scaling},  with $T^*(x)$ given in \eqref{18dec1},
 $K_2$ depends on $ p, a,  \delta_{0}(x_{0})$, 
 $T(x_{0})$, $t_0(x_0)$ and 
$\|(u(t_0(x_0)),\partial_tu(t_0(x_0)))\|_{H^{1}\times
L^{2}(I(x_0,\frac{T(x_0)-t_0(x_0)}{\delta_0(x_0)}) )}$.
\end{prop}
\begin{nb}
By using  the Sobolev's embedding in one dimension space and the above proposition, we can deduce  that 
\begin{equation}\label{16dec5bis}
\|w(s)\|_{L^{\infty}(-1,1)} \le Ks^{q_1},\quad  \textrm{for all }\ \   s\geq -\log (T^*(x)-t_{0}(x_{0})).
\end{equation}
\end{nb}

\medskip

  {\it{Proof of Proposition \ref{prop3.1}:}}
We proceed in 2 steps:\\
-In step 1, we use   the covering technique  and  the   Sobolev's embedding in two dimensions (space-time)    to conclude a polynomial estimate related
to the $L^{p+2}(-1,1)$ norm  of   $w(s)$.\\
-In step 2, by exploiting the result obtained in  step 1 and the fact that $N_{m_0,1}(w(s),s)$  (defined in \eqref{lyap1})  is a decreasing functional, we easily conclude   the estimate \eqref{co2bis} .

\medskip

{\bf Step 1:} 
By using Theorem\ref{tp1}',
we  get  for all $s\ge  -\log (T^*(x)-t_0(x_0)),$ 
\begin{equation}\label{a1bis}
\int_{s}^{s+1} \!\!\iint\!\!\Big((\partial_s
w(y,\tau))^2+(\partial_y
w(y,\tau))^2
+w^2 (y,\tau)\Big){\mathrm{d}}y{\mathrm{d}}\tau\le
K_1  s^{q}.
\end{equation}

\medskip

Now, we use   the   Sobolev's embedding in two dimensions (space-time)    
 and \eqref{a1bis}   to conclude a polynomial estimate related
to the $L^{p+2}(-1,1)$ norm  of   $w(s)$.
 Indeed,  for all $s\ge  -\log (T^*(x)-t_0(x_0))$, 
 by using the mean value theorem, we derive the existence of $\sigma(s)\in [s,s+1]$ such that
\begin{equation}\label{s1}
\iint\!|w(y,\sigma (s))|^{p+2}{\mathrm{d}}y
=\int_{s}^{s+1}\iint\!|w(y,\tau
)|^{p+2}{\mathrm{d}}y{\mathrm{d}}\tau.
\end{equation}
 Let us  
 write  the identity for all $s\ge  -\log (T^*(x)-t_0(x_0)),$  
\begin{align}\label{2018d1}
\iint\!|w(y,s )|^{p+2}{\mathrm{d}}y=&\iint\!|w(y,\sigma (s))|^{p+2}{\mathrm{d}}y  +
\int_{\sigma (s)}^{s}\frac{d}{d \tau}\iint |w(y,\tau)|^{p+2}{\mathrm{d}}y{\mathrm{d}}\tau.
\end{align}
By combining (\ref{s1}), \eqref{2018d1} and \eqref{basic1}, we infer  for all $s\ge  -\log (T^*(x)-t_0(x_0)),$ 
\begin{align}\label{janv121}
\iint\!|w(y,s )|^{p+2}{\mathrm{d}}y
\le& \int_{s}^{s+1}\iint\!|w(y,\tau
)|^{p+2}{\mathrm{d}}y{\mathrm{d}}\tau
 + C\int_{s}^{s+1}\iint\!|w(y,\tau)|^{2p+2}{\mathrm{d}}y{\mathrm{d}}\tau\no\\
&+C\int_{s}^{s+1}\iint\!(\partial_s
w(y,\tau))^2{\mathrm{d}}y{\mathrm{d}}\tau.
\end{align}
By using Sobolev's inequalities  in two dimension (space time) and \eqref{a1bis}, we conclude  that
\begin{equation}\label{a6}
\int_{s}^{s+1} \!\!\iint\!\!|w (y,\tau)|^{2p+2}\y{\mathrm{d}}\tau\le
C \Big(\int_{s}^{s+1} \!\!\iint\!\!\Big((\partial_s
w)^2+(\partial_yw)^2+w ^{2}\Big){\mathrm{d}}y{\mathrm{d}}\tau\Big)^{p+1}\le
K s^{q(p+1)}.
\end{equation}
Due to  the classical inequality $x^{p+2}\le 1+x^{2p+2},$ for all $x\ge 0$, we have
\begin{align}\label{2018s1}
\int_{s}^{s+1}\iint\!|w(y,\tau
)|^{p+2}{\mathrm{d}}y{\mathrm{d}}\tau
\le&  C
+C\int_{s}^{s+1}\iint\!|w(y,\tau
)|^{2p+2}{\mathrm{d}}y{\mathrm{d}}\tau.
\end{align}
By combining (\ref{janv121}), \eqref{a6}, \eqref{2018s1} and  \eqref{a1bis}, we deduce  for all $s\ge  -\log (T^*(x)-t_0(x_0)),$  that
\begin{equation}\label{s1jan19}
\iint\!|w(y,s )|^{p+2}{\mathrm{d}}y
\le K s^{q(p+1)}.
\end{equation}

\medskip

{\bf Step 2:}  From  (\ref{equiv1}), \eqref{id1}, 
this yields
\begin{equation}\label{BB6}
e^{-\frac{2(p+1)s}{p-1}}s^{\frac{2a}{p-1}}  \iint F(\p w)\w \y \le \frac{C}{s^a}  \iint |w(y,s )|^{p+1}\log^a (2+\p^2w^2)  \y+Ce^{-2s}.
\end{equation}
To estimate the right-hand  side in the inequality \eqref{BB6}, we consider two cases: \\
{\bf Case 1:} the  case where $a\ge 0$.\\
From this inequality $ 2+x^2y^2\le  (2+x^2)(2+y^2),$  for all $x, y \in \er$, and the fact that  $\log^a$ is an increasing function on the interval $[2,\infty)$, we conclude that 
\begin{equation}\label{0B2}
\log^a (2+x^2y^2)\le \Big(\log (2+x^2)+\log (2+y^2)\Big)^a.
\end{equation}
Using the inequality $(X+Y)^a\le C(X^a+Y^a),$  for all $X, Y \in \er_+$ and \eqref{0B2}, we obtain
\begin{equation}\label{B2}
\log^a (2+x^2y^2)\le C\log^a (2+x^2)+C\log^a (2+y^2).
\end{equation}
By combining \eqref{B2} and   the inequality $\log^a (2+z^2)\le C+C|z|$,  for all $z\in \er$, we conclude that
\begin{equation}\label{00B2}
\log^a (2+x^2y^2)\le C\log^a (2+x^2)+C+C|y|.
\end{equation}
Hence, by taking into account   \eqref{00B2} and  \eqref{defphi},   we  deduce that
\begin{align}\label{12jan5}
 \frac{1}{s^a}  \iint |w|^{p+1}\log^a (2+\p^2 w^2)  \y
&\le C+C    \iint |w|^{p+1}  \y
+\frac{C}{s^a}  \iint |w|^{p+2} \y.
\end{align}
Therefore, using \eqref{s1jan19}, \eqref{BB6}, \eqref{12jan5} and   Jensen's inequality,   we get
\begin{equation}\label{BB6bis}
e^{-\frac{2(p+1)s}{p-1}}s^{\frac{2a}{p-1}}  \iint F(\p w)\w \y \le K s^{q(p+1)}.
\end{equation} 
{\bf Case 2:} the  case where $a<0$.\\
 Using \eqref{BB6}, we get 
\begin{equation}\label{BB7}
e^{-\frac{2(p+1)s}{p-1}}s^{\frac{2a}{p-1}}  \iint F(\p w)\w \y \le \frac{C}{s^a}  \iint |w(y,s )|^{p+1}  \y+Ce^{-2s}.
\end{equation}
By  Jensen's inequality and \eqref{s1jan19}, we conclude that 
\begin{equation}\label{BB8}
e^{-\frac{2(p+1)s}{p-1}}s^{\frac{2a}{p-1}}  \iint F(\p w)\w \y \le K s^{q(p+1)-a}.
\end{equation}
Thanks to
  \eqref{BB6bis} and   \eqref{BB8} ,  we deduce  for all $a\in \er$,  $s\ge  -\log (T^*(x)-t_0(x_0)),$
\begin{equation}\label{BB818}
e^{-\frac{2(p+1)s}{p-1}}s^{\frac{2a}{p-1}}  \iint F(\p w)\w \y \le K s^{q(p+1)+|a|}.
\end{equation}

Now, we use    \eqref{2018N0}, 
 \eqref{2018N1},    \eqref{defb}, the fact that $b+1=q$  and   the definition of $N_{m_0,1}(w(s),s)$ defined  in \eqref{lyap1},  to conclude  for all $s\ge  -\log (T^*(x)-t_0(x_0)),$
\begin{equation}\label{lyap1806}
H_{m_0,1}(w(s),s)\le Ks^{b}\le K s^q,
\end{equation}
 where $H_{m_0,1}$ is defined in \eqref{F0}.
Thanks to
  \eqref{BB818} and   the definition of $H_{m_0,1}(w(s),s)$,  we deduce  for all $s\ge  -\log (T^*(x)-t_0(x_0)),$
\begin{equation}\label{2018d5}
\iint \Big((\partial_{s}w)^2+(\partial_y w)^2(1-y^2)+w^2\Big)\w \y \le Ks^{q(p+1)+|a|}.
\end{equation}

%
Note that the estimate \eqref{2018d5} implies  \ref{co2bis} (take $q_1={q(p+1)+|a|}$)  but just  in $(-\frac12, \frac12)$. 
By using   the covering technique, we extend this estimate from  $(-\frac12, \frac12)$ to $(-1, 1)$,  we refer the reader to Merle
 and Zaag \cite{MZimrn05} (unperturbed case) and Hamza and Zaag 
\cite{HZjhde12}  (perturbed case).
This concludes  the Proposition \ref{prop3.1}. \Box

\subsection{A Lyapunov functional}\label{3.2}
In this subsection, our aim is to construct a Lyapunov functional for equation \eqref{A}.
 Note that this functional is far from being trivial and makes our main contribution.
More precisely, thanks to the  rough estimate obtained in  the Proposition \ref{prop3.1}, 
we derive here that the functional  $L(w(s),s)$ defined in \eqref{10dec2} is a decreasing 
  functional  of time  for equation (\ref{A}),  provided that is   $s$ large enough. \\

Let us remark that in   Section \ref{section2}, we construct a Lyapunov functional $N_{m_0,1}(w(s),s)$ defined in \eqref{lyap1}, but we obtain just a rough estimate because the multiplier is not bounded. Nevertheless, the multiplier  related to  the functional  $L(w(s),s)$
 is nonnegative and bounded. Then,  as we said above,   the natural energy
  $E_{1}(w(s),s)$ defined in \eqref{5jan1} is  a small perturbation of  $L(w(s),s)$.

\medskip
Consider $u$ a solution of \eqref{gen}  with blow-up graph  $\Gamma:\{x\mapsto T(x)\}$ and  $x_0$  a 
non characteristic point.
Let   $T_0\in (t_0(x_0), T(x_0)]$.  For all 
$x\in \er$  such that $|x-x_0|\le \frac{T_0-t_0(x_0)}{\delta_0(x_0)}$,  we write $w$ instead of $w_{x,T^*(x)}$ defined in (\ref{scaling}) with $T^*(x)$ given by 
 \eqref{18dec1}.
Thanks to estimate \eqref{co2bis},  we can improve estimate \eqref{E011} related to the 
 control of  the time derivative of the  functional $E_{1}(w(s),s)$. More precisely, we prove  the following lemma:
\begin{lem}\label{2018lem31}  For all   
  $s \geq -\log (T^*(x)-t_0(x_0))$, we have  
\begin{align}\label{E01}
\frac{d}{ds}E_{1}(w(s),s)\le &-  \frac{3}{p-1}\iint (\partial_{s}w)^2\frac{\w}{1-y^2}\y\\
&+ 
 \frac{K\log s}{s^{a+2}}\iint |w|^{p+1}\log^a(2+\p^2w^2)\w \y\no\\
&+ \frac{C}{s^2}\iint  (\partial_y w)^2(1-y^2)\w \y+\frac{C}{s^2}\iint  w^2\w \y  +C e^{-s}.
 \end{align}
\end{lem}

{\it Proof}:  Since we  consider the one space dimension and  by using the  additional information obtained in Subsection \ref{3.1}, we are going to refine the estimate related to $\Sigma_{1}^{1}(s)$ and $\Sigma_{1}^{2}(s)$ defined in \eqref{E00}. Let us mention that the estimate  
\eqref{sigma13} related to $\Sigma_{1}^{3}(s)+\Sigma_{1}^{4}(s)$ defined in \eqref{E00} is acceptable and does not need any improvement.
 More precisely,  we write 
\begin{align}
\Sigma_{1}^{1}(s)+\Sigma_{1}^{2}(s)=&\frac{2p+2}{p-1}
e^{-\frac{2(p+1)s}{p-1}}s^{\frac{2a}{p-1}}\iint\big( F(\p w)-\frac{\ps wf(\p w)}{p+1}\big)\w \y\no\\
&-\frac{2a}{p-1}e^{-\frac{2(p+1)s}{p-1}}s^{\frac{2a}{p-1}-1}\iint \big( F(\p w)-\frac{\p wf(\p w)}{2}\big)\w \y\no.
 \end{align}
We  attempt to group  the main terms together. 
A straightforward computations implies  that
\begin{equation}\label{2018id2}
\Sigma_{1}^{1}(s)+\Sigma_{1}^{2}(s)=\chi_1(s)+\chi_2(s),
\end{equation}
where
\begin{align}
\chi_{1}(s)=&
\frac{a}{(p+1)s^{a+1}}\iint  {|  w|^{p+1}}\log^{{a}}(2+\p^2 w^2  )\Big(1 -\frac{4s}{(p-1)\log (2+\p^2w^2 )}\Big)\w \y,
\label{2018id5}\\
\chi_{2}(s)=&\frac{2e^{-\frac{2(p+1)s}{p-1}}}{p-1}s^{\frac{2a}{p-1}}\iint \Big((p+1) F_2(\p w)-\frac{a}{s} F_1(\p w)-\frac{a}{s}F_2(\p w)\Big)\w \y,\label{2018id8}
\end{align}
$F_1$ and $F_2$ are defined by
\begin{equation}
F_1(x)= -\frac{ 2a} {(p+1)^2}{| x|^{p+1}}\log^{{a-1}}(2+x^2  ),\label{defF2}
\end{equation}
and 
\begin{equation}\label{defF123}
F_2(x)=F(x)-\frac{xf(x)}{p+1}-F_1(x).
\end{equation}
\medskip

We would like now to find  an estimate for the term $\chi_{1}(s)$. For this,  for all  $s\geq  -\log (T^*(x)-t_0(x_0))$,
we divide $(-1,1)$ into two parts
 \begin{equation}\label{27nov1}
A_{1}(s)=\{y \in (-1,1)\,\,|\,\, \ps w^2(y,s)\leq  1\}\,\,{\rm and }\,\,A_{2}(s)=\{y \in (-1,1)
\,\,|\,\, \ps w^2(y,s)\ge  1\}.
\end{equation}
Accordingly, we write  $\chi_1(s)=\chi_1^1(s)+\chi_1^2(s)$, where
\begin{align}
\chi_1^{1}(s)=&
\frac{a}{(p+1)s^{a+1}}\int_{A_1(s)}   {|  w|^{p+1}}\log^{{a}}(2+\p^2 w^2  )\Big(1 -\frac{4s}{(p-1)\log (2+\p^2 w^2 )}\Big)\w \y,\label{130}\\
\chi_1^{2}(s)=&
\frac{a}{(p+1)s^{a+1}}\int_{A_2(s)}   {|  w|^{p+1}}\log^{{a}}(2+\p^2 w^2  )\Big(1 -\frac{4s}{(p-1)\log (2+\p^2 w^2 )}\Big)\w \y.\label{131}
\end{align}
Note that, by using 
 the definition of the set $A_1(s)$ given  in \eqref{27nov1},  we get, for all $s\geq  -\log (T^*(x)-t_0(x_0)),$
\begin{equation}\label{16dec1}
 | w|^{p+1}\log^{{a}}(2+\p^2 w^2  )\le C \p^{-\frac{p+1}2}(s)\log^a(2+\ps) \le Ce^{-s}.
\end{equation}
From \eqref{16dec1}  and  the fact that 
\begin{equation}\label{16dec2}
 1-\frac{4s}{(p-1)\log (2+\p^2 w^2)}
 \le C,
\end{equation}
 we get
\begin{equation}\label{93}
\chi^1_1(s) \le Ce^{-s}.
\end{equation}
Next, by using  the definition of the set $A_2(s)$ defined in \eqref{27nov1}, we write  for all $s\geq  -\log (T^*(x)-t_0(x_0)),$
\begin{equation}\label{16dec10}
 1-\frac{4s}{(p-1)\log (2+\p^2 w^2)}=\frac{1}{\log (2+\p^2 w^2)}
\Big({\log (2+\p^2 w^2)}
 -\frac{4s}{p-1}\Big).
\end{equation}
Here,  the estimate proved in Subsection \ref{3.1} is crucial to conclude. More precisely,  
by exploiting the expression of $\p$ given in \eqref{psi} and  the  estimate  \eqref{16dec5bis}, we conclude that 
\begin{equation} \label{16dec11}
 {\log (2+\p^2 w^2)}
 -\frac{4s}{p-1}\le K \log s.
\end{equation}
Also, by using the definition of the set $A_2(s)$ defined in \eqref{27nov1}, we can write 
 for all $ s \geq  -\log (T^*(x)-t_0(x_0)),$   if $y\in A_{2}(s)$, we have 
\begin{equation}\label{16dec12}
\log(2+\p^2w^2)\ge \log(\ps)\ge \frac{2s}{p-1}-\frac{a \log s}{p-1}.
\end{equation}
By using  \eqref{16dec10},  \eqref{16dec11} and  \eqref{16dec12} we have   for all $ s \geq  -\log (T^*(x)-t_0(x_0)),$
\begin{equation}\label{16dec13}
 1-\frac{4s}{(p-1)\log (2+\p^2 w^2)}\le K\frac{\log s}{s}.
\end{equation}
Adding  \eqref{16dec13} and \eqref{131}, we have 
\begin{equation}\label{16dec14}
\chi_1^{2}(s)\le 
\frac{K\log s}{s^{a+2}}\iint   {|  w|^{p+1}}\log^{{a}}(2+\p^2 w^2 )\w \y.
\end{equation}
Note that, by using  the fact $\chi_1(s)=\chi_1^{1}(s)+\chi_1^{2}(s)$, \eqref{93} and \eqref{16dec14}, we get
\begin{equation}\label{13janva1}
\chi_1(s)\le 
\frac{K\log s}{s^{a+2}}\iint   {|  w|^{p+1}}\log^{{a}}(2+\p^2 w^2 )\w \y+Ce^{-s}.
\end{equation}
Finally, it remains only to control the term $\chi_2(s)$.
Note from    \eqref{equiv2}  and  \eqref{equiv3}  that
\begin{equation}\label{15dec1}
\frac{1}{s}| F_1(\p w)|+| F_2(\p w)|\le   C+C \frac{ \p w}{s^2}f(\p w).
\end{equation}
By \eqref{2018id8}, \eqref{15dec1} and \eqref{id1},  we have, for all $s\ge  -\log (T^*(x)-t_0(x_0))$,
\begin{equation}\label{sigma11dec18}
\chi_2(s)\le 
\frac{C}{s^{a+2}}\iint |w|^{p+1}\log^a(2+\p^2w^2)\w \y+  C e^{-2s}.
\end{equation}
The result \eqref{E01} derives immediately from  \eqref{E00}, \eqref{sigma13},  \eqref{13janva1}, \eqref{sigma11dec18},  and  the identity \eqref{2018id2},
which ends the proof of Lemma \ref{2018lem31}
\Box

%
%
%
%
%

\medskip

With Lemmas \ref{LemJ_0} and  \ref{2018lem31},  we are in a position to state and prove
Theorem \ref{t1}', which is a uniform version of Theorem \ref{t1} for $x$ near $x_0$.


\bigskip

\noindent {\bf{Theorem} \ref{t1}'} {\it (Existence  of a Lyapunov functional for equation }
{\eqref{A}})\\
\label{t1bis}
{\it
Consider   $u $    a solution of ({\ref{gen}}) with blow-up graph
$\Gamma:\{x\mapsto T(x)\}$ and  $x_0$  a non characteristic point.
Then there exists $t_1(x_0)\in [0,T(x_0)) $ such that, 
 for all $T_0\in  (t_1(x_0),T(x_0)]$,  for all  $s\ge  -\log(T_0-t_1(x_0))$  and $x\in \er$, where $|x-x_0|\le \frac{e^{-s}}{\delta_0(x_0)}$,
 we have
\begin{equation}\label{t1lyap}
 L(w(s+1),s+1)-L(w(s),s) \leq -\frac{2}{p-1} \int_{s}^{s+1}\iint (\partial_{s}w)^2\frac{\w}{1-y^2}\y \t,
\end{equation}
where  $w=w_{x,T^*(x)}$ and $T^*(x)$  is defined in \eqref{18dec1}.
}

\medskip

  {\it{Proof of Theorem \ref{t1}':}}
By exploiting the defintion of $L_0(w(s),s)$ in \eqref{5jan1},
we  can write easily
\begin{equation}\label{14jan1}
\frac{d}{ds}L_0(w(s),s)=\frac{d}{ds}E_1(w(s),s) + \frac{1}{\sqrt{s}}\frac{d}{ds}J_1(w(s),s)-  \frac{1}{2s\sqrt{s}}J_1(w(s),s),
\end{equation}
where $J_1(w(s),s)=\frac1{s}\iint   w\partial_{s}w\w \y$.
Lemmas \ref{LemJ_0} and  \ref{2018lem31} and 
the following inequality
$$
\frac{1}{2s^2\sqrt{s}}\iint   w\partial_{s}w\w \y+\frac{p+3}{2s^3}\iint   w\partial_{s}w\w \y\le 
\frac{C}{s^2}\iint   (\partial_{s}w)^2\w \y+\frac{C}{s^2}\iint   w^2\w \y,$$
 allows to prove that 
 for all $s \geq -\log(T^*(x)-t_0(x_0))$, we have 
\begin{align*}
\frac{d}{ds}L_0(w(s),s)\le &-  (\frac{3}{p-1}- \frac{C}{s})\iint (\partial_{s}w)^2\frac{\w}{1-y^2}\y+ \frac{p+3}{2s\sqrt{s}}L_0(w(s),s)\no\\
&- \frac{1}{s\sqrt{s}}( \frac{p+1}{2(p-1)} -\frac{C}{\sqrt{s}})\iint  w^2\w \y  \\
&-\frac{1}{s\sqrt{s}}(\frac{p+7}{4}-\frac{C}{\sqrt{s}})\iint   (\partial_{s}w)^2\w \y\\
&-\frac1{s\sqrt{s}}(\frac{p-1}{4}-\frac{C}{\sqrt{s}})\iint (\partial_y w)^2(1-y^2)\w\y
\no\\
&-\frac1{s^{a+\frac32}}(\frac{p-1}{2(p+1)}-  \frac{K\log s}{\sqrt{s}}-\frac{C}{s})\iint |w|^{p+1}\log^a(2+\p^2w^2)  
\w\y\\
&+ C \frac{e^{-2s}}{\sqrt{s}}+C e^{-s}.
  \end{align*}
Again,  choosing  $S_3> -\log(T(x_0)-t_0(x_0))$ large enough,  this  implies that  for all $s \geq \max  (-\log(T^*(x)-t_0(x_0)), S_{3})$, we have 
\begin{equation}\label{17dec1}
\frac{d}{ds}L_0(w(s),s)\le -\frac{2}{p-1}\iint (\partial_{s}w)^2\frac{\w}{1-y^2}\y+ \frac{p+3}{2s\sqrt{s}}L_0(w(s),s)+  
C e^{-s}.
\end{equation}
Recalling that,
$$L(w(s),s)=\exp\Big(\frac{p+3}{\sqrt{s}}\Big) L_0(w(s),s)+\theta e^{-{s}},$$  
   we get from straightforward computations 
  \begin{equation}\label{17dec2}
  \frac{d}{ds}L(w(s),s) =-\frac{p+3}{2s\sqrt{s}}\exp\Big(\frac{p+3}{\sqrt{s}}\Big) L_0(w(s),s)+\exp\Big(\frac{p+3}{\sqrt{s}}\Big)\frac{d}{ds}L_0(w(s),s)-{\theta} e^{-{s}}.
\end{equation}
Therefore, estimates $\eqref{17dec1}$ and $\eqref{17dec2}$ lead to the following crucial estimate:
 \begin{equation}
 \frac{d}{ds}L(w(s),s)\leq -\frac2{p-1} \exp\Big(\frac{p+3}{\sqrt{s}}\Big) \iint (\partial_{s}w)^2 \frac{\w}{1-y^2}\y
+\Big(C\exp\Big(\frac{p+3}{\sqrt{s}}\Big)-\theta \Big)e^{-s}.
\end{equation}
Since we have $1\leq \exp\Big(\frac{p+3}{\sqrt{s}}\Big)\leq \exp\Big(\frac{p+3}{\sqrt{S_3}}\Big)$, 
we then choose $\theta$ large enough, so that $C-\theta \leq 0$, which yields,  for all $s \geq \max  (-\log(T^*(x)-t_0(x_0)), S_{3})$,
$$\frac{d}{ds}L(w(s),s)\leq -\frac2{p-1} \iint (\partial_{s}w)^2 \frac{\w}{1-y^2}\y.$$
A simple integration between $s$ and $s+1$ ensures the result
\eqref{t1lyap}, where 
\begin{equation}\label{new19dec1}
 t_1(x_0)=\max(T(x_0)-e^{-S_3},t_0(x_0)).
\end{equation}
This concludes the proof
of  Theorem \ref{t1}'.
\Box

\medskip

 We now claim the following lemma:
\begin{lem}\label{L19}
 There exists $S_4\ge S_3$ such that, if
  $L(w(s_3),s_3)<0$ for some $s_3\ge \max(S_4,-\log (T^*(x)-t_1(x_0)))$, then $w$
blows up in some finite time $s_4>s_3$.
\end{lem}
 {\it Proof:}
The argument is the same  as  the similar part in  Proposition \ref{proplyap} in this paper.

\Box

\vspace{0.1cm}

\subsection{Proof of Theorem \ref{t2} }\label{3.3}

In this subsection, we prove Theorem \ref{t2}.  Note that the lower bound follows from the finite speed of propagation and the wellposedness  in $H^1\times L^2$.
For a detailed argument in the similar case of equation (\ref{NLW}), see Lemma 3.1 (page 1136) in \cite{MZimrn05}.\\
\noindent We consider $u$  a solution of (\ref{gen}) which is defined under the graph of $x\mapsto T(x)$, and $x_0$ a non characteristic point.
Let 
\begin{equation}\label{21dec1}
 t_2(x_0)=\max(T(x_0)-e^{-S_4},t_1(x_0)).
\end{equation}
Given some $T_0\in (t_2(x_0),T(x_0)]$, for all
$x\in \er$ is such that $|x-x_0|\le \frac{T_0-t_2(x_0)}{\delta_0(x_0)}$, where $\delta(x_0)$ is defined in (\ref{nonchar}), we aim   at bounding $\|(w,\partial_s w)(s)\|_{H^1\times L^2((-1,1))}$ for $s$ large.

\medskip


\noindent As in  \cite{ HZnonl12, omar1}, by combining Theorem  \ref{t1}' 
 and  Lemma \ref{L19}  we get the following bounds:
\begin{cor}\label{22jan1} { (Bound on  $L_0(w(s),s)$)}.
For all $T_0\in (t_2(x_0),T(x_0)]$, for all $s\ge  -\log (T_0-t_2(x_0))$
 and $x\in \er$ where $|x-x_0|\le \frac{e^{-s}}{\delta_0(x_0)}$,
 we have
\begin{equation}\label{2018N00}
-C\leq L_{0}(w(s),s) \leq C L_{0}(w(\tilde{s}_{2}),\tilde{s}_{2})+C ,
\end{equation}
where $\tilde{s}_{2}=-\log (T^*(x)-t_2(x_0))$.\\
Moreover, 
 for all $s\ge  -\log (T^*(x)-t_2(x_0))$, we have
\begin{equation}\label{22janv2}
\int_{s}^{s+1}
 \iint (\partial_{s}w)^2 \frac{\w}{1-y^2}\y \s \le K,
\end{equation}
where  $K=K(a,p,T^*(x),\|(u(t_2),u_t(t_2))\|_{
H^{1}\times L^{2}(I(x_0,\frac{T_0-t_2(x_0)}{\delta_0(x_0)})
)})$,
$C=C(a,p)$    and   $\delta_0(x_0)\in (0,1)$ is defined in (\ref{nonchar}).
\end{cor}
\medskip

\begin{nb} Using the definition of  (\ref{scaling}) of
$w_{x,T^*(x)}=w$, we write easily
\begin{equation}\label{2018N5}
   L_0(w(\widetilde{s_2}), \widetilde{s_2})
\le \widetilde{K_1},
\end{equation}
where
$\widetilde{K_1}=\widetilde{K_1}(T(x_0)-t_2(x_0),\|(u(t_2(x_0)),\partial_tu(t_2(x_0)))\|_{H^{1}\times L^{2}(I(x_0,\frac{T(x_0)-t_2(x_0)}{\delta_0(x_0)}))})$.
\end{nb}

\medskip

\noindent Starting from these bounds, the proof of Theorem \ref{t2} is similar to the proof in \cite{MZajm03,MZimrn05}
except for
the treatment of the nonlinear terms and
of the perturbation terms. In our opinion, handling these terms is straightforward in
all the steps of the proof, except for the first step, where we bound the time averages of the $L^{p+1}_{\rho}((-1,1))$ norm of $w$. For that reason,
we only give that step and refer to  \cite{MZajm03,MZimrn05} for the remaining steps in the proof of  Theorem \ref{t2}. This is the step we prove here.
\begin{prop}\label{projan22}
 For all  $s\ge 1-\log (T^*(x)-t_2(x_0))$,
\begin{equation}\label{projan221}
\int_{s}^{s+1}\!\!\iint e^{-\frac{2(p+1)s}{p-1}}s^{\frac{2a}{p-1}} F(\p w)\w {\mathrm{d}}y\t
\le K.
\end{equation}
\end{prop}

Proof:
For $s\ge 1-\log (T^*(x)-t_2(x_0))$, let us work with time integrals betwen $s_1$ et $s_2$ where $s_1\in [s-1,s]$
and $s_2\in [s+1,s+2]$.
 By integrating the expression (\ref{5jan1}) of $L_0(w(s),s)$ in time between $s_1$ and $s_2$, where $s_2>s_1>-\log (T^*(x)-t_2(x_0))$, we obtain:
\begin{align}\label{et}
\int_{s_1}^{s_2}\!\!L_0(w(s),s)\s=&\displaystyle\int_{s_1}^{s_2}\iint\!\!\Big ( \frac{1}{2}(\partial_{s}w)^2
+\frac{p+1}{(p-1)^2}w^2
-e^{-\frac{2(p+1)s}{p-1}}s^{\frac{2a}{p-1}}   F(\p w)\Big )\w {\mathrm{d}}y{\mathrm{d}}s\nonumber\\
&+\frac{1}{2}\displaystyle\int_{s_1}^{s_2}\!\!\iint\!\!(\partial_y w)^2(1-y^2)\w {\mathrm{d}}y{\mathrm{d}}s-\int_{s_1}^{s_2}\!\! \frac1{s\sqrt{s}}\!\!\displaystyle\iint\!\!w\partial_s w\w {\mathrm{d}}y{\mathrm{d}}s.
\end{align}
By multiplying the equation (\ref{A}) by $w\w$ and integrating both in time and in space over $(-1,1)\times [s_1,s_2]$  we obtain the following identity, after some
integration by parts:
%
\begin{eqnarray}\label{et1}
&&\Big [\iint\!\!\Big (w\partial_{s}w+\frac{5-p}{2(p-1)}w^2\Big ) \w{\mathrm{d}}y\Big ]_{s_1}^{s_2}=
\int_{s_1}^{s_2}\!\!\iint\!\!(\partial_{s}w)^2\w{\mathrm{d}}y{\mathrm{d}}s\\
&&-\int_{s_1}^{s_2}\!\!\iint\!\!(\partial_y w)^2(1-y^2)\w{\mathrm{d}}y{\mathrm{d}}s
-\frac{2p+2}{(p-1)^2}\int_{s_1}^{s_2}\!\!\iint\!\!w^2\w{\mathrm{d}}y{\mathrm{d}}s\nonumber\\
&&+\int_{s_1}^{s_2}\!\!\iint\!\!e^{-\frac{2ps}{p-1}}s^{\frac{a}{p-1}}w f(\p w)\w{\mathrm{d}}y{\mathrm{d}}s-\frac4{p-1}\!\int_{s_1}^{s_2}\!\!\iint\!\!w\partial_{s}w
\frac{y^2\w}{1-y^2}{\mathrm{d}}y{\mathrm{d}}s\nonumber\\
&&+2\!\!\int_{s_1}^{s_2}\!\!\iint\!\!y\partial_{y}w\partial_{s} w \w{\mathrm{d}}y{\mathrm{d}}s
+\frac{2a}{p-1}\int_{s_1}^{s_2}\iint \frac1{s}y\partial_y w w\w\y\s\nonumber\\
&&+\!\!\int_{s_1}^{s_2}\!\!\iint\!\!\gamma(s) w^2\w\y\s
+\frac{2a}{p-1}\int_{s_1}^{s_2}\iint \frac{1}{s}\partial_s w w\w \y\s.\nonumber
\end{eqnarray}
Note that, by using the identity  \eqref{defF123}, we get
\begin{align}\label{23jan12}
e^{-\frac{2(p+1)s}{p-1}}s^{\frac{2a}{p-1}}\Big(\frac{\p w}2 f(\p w)- F(\p w)\Big)&=\frac{p-1}{2}e^{-\frac{2(p+1)s}{p-1}}s^{\frac{2a}{p-1}} F(\p w)\\
&-\frac{p+1}2e^{-\frac{2(p+1)s}{p-1}}s^{\frac{2a}{p-1}}\Big(F_1(\p w)+F_2(\p w)\Big).\nonumber
\end{align}
By combining the identities (\ref{et}),  (\ref{et1}) and exploiting  \eqref{23jan12}, we obtain
\begin{eqnarray}\label{et44}
&&\frac{p-1}{2}\int_{s_1}^{s_2}\!\!\iint e^{-\frac{2(p+1)s}{p-1}}s^{\frac{2a}{p-1}} F(\p w)\w {\mathrm{d}}y{\mathrm{d}}s\nonumber\\
&=&\frac12\Big [\iint\!\!\Big (w\partial_{s}w+\frac{5-p}{2(p-1)}w^2\Big ) \w{\mathrm{d}}y\Big ]_{s_1}^{s_2}-
\int_{s_1}^{s_2}\!\!\iint\!\!(\partial_{s}w)^2\w{\mathrm{d}}y{\mathrm{d}}s\nonumber\\
&&+\int_{s_1}^{s_2}\!\!L_0(w(s),s)\s+\frac2{p-1}\!\int_{s_1}^{s_2}\!\!\iint\!\!w\partial_{s}w
\frac{y^2\w}{1-y^2}{\mathrm{d}}y{\mathrm{d}}s\nonumber\\
&&-\!\!\int_{s_1}^{s_2}\!\!\iint\!\!y\partial_{y}w\partial_{s} w \w{\mathrm{d}}y{\mathrm{d}}s
\underbrace{-\frac{a}{p-1}\int_{s_1}^{s_2}\iint \frac1{s}y\partial_y w w\w\y\s}_{A_1}\nonumber\\
&&\underbrace{-\frac12 \int_{s_1}^{s_2}\!\!\iint\!\!\gamma(s) w^2\w\y\s}_{A_2}
\underbrace{-\frac{a}{p-1}\int_{s_1}^{s_2}\iint \frac{1}{s}\partial_s w w\w \y\s}_{A_3}\\
&&+\underbrace{\int_{s_1}^{s_2}\!\! \frac1{s\sqrt{s}}\!\!\displaystyle\iint\!\!w\partial_s w\w {\mathrm{d}}y{\mathrm{d}}s}_{A_4}+\underbrace{\frac{p+1}{2}\int_{s_1}^{s_2}\!\!\iint e^{-\frac{2(p+1)s}{p-1}}s^{\frac{2a}{p-1}} F_1(\p w)\w {\mathrm{d}}y{\mathrm{d}}s}_{A_5}\nonumber\\
&&+\underbrace{\frac{p+1}{2}\int_{s_1}^{s_2}\!\!\iint e^{-\frac{2(p+1)s}{p-1}}s^{\frac{2a}{p-1}} F_2(\p w)\w {\mathrm{d}}y{\mathrm{d}}s}_{A_6}.\nonumber
\end{eqnarray}
We claim that Proposition \ref{projan22} follows from the following Lemma where we    control  the space-time integral of the nonlinear term  of $w$ and all the terms on the right-hand side of the relation
 (\ref{et44}) in terms of
the left-hand  side:
\begin{lem}\label{g}
 For all  $s\ge 1-\log (T^*(x)-t_3(x_0))$,
for some $t_3(x_0)\in [t_2(x_0), T(x_0))$, for all $\varepsilon>0$,
\begin{equation}\label{control}
\!\!\iint  |w|^{p+1}\w {\mathrm{d}}y
\le K +C
\iint e^{-\frac{2(p+1)s}{p-1}}s^{\frac{2a}{p-1}} F(\p w)\w {\mathrm{d}}y,
\end{equation}
\begin{equation}\label{0control}
\iint e^{-\frac{2(p+1)s}{p-1}}s^{\frac{2a}{p-1}} F(\p w)\w {\mathrm{d}}y
\le   K +C
\!\!\iint  |w|^{p+1}\w {\mathrm{d}}y,
\end{equation}
\begin{equation}\label{control3}
\int_{s_1}^{s_2}\!\!\!\iint\!|y\partial_yw\partial_s w| \w{\mathrm{d}}y{\mathrm{d}}s
\le \frac{K}{\varepsilon} +K\varepsilon
\N,\qquad
\end{equation}
\begin{equation}\label{control1}
\sup_{s\in [s_1,s_2]}\iint \!\!w^2(y,s)\w{\mathrm{d}}y\le \frac{K}{\varepsilon} +K\varepsilon\N,
\end{equation}
\begin{equation}\label{control30}
\!\int_{s_1}^{s_2}\!\!\iint\!\!w\partial_{s}w
\frac{y^2\w}{1-y^2}{\mathrm{d}}y{\mathrm{d}}s
\le \frac{K}{\varepsilon} +K\varepsilon
\N,\qquad
\end{equation}
\begin{align}\label{control4}
\iint|w\partial_{s}w|\w {\mathrm{d}}y\le&
\iint (\partial_{s}w)^2\w {\mathrm{d}}y+\frac{K}{\varepsilon}\nonumber\\ 
&+K\varepsilon
\N,
\end{align}
\begin{equation}\label{control5}
\iint \Big( (\partial_{s}w(y,s_1))^2+(\partial_{s}w(y,s_2))^2\Big)\w {\mathrm{d}}y
\le K,
\end{equation}
\begin{equation}\label{A10}
|A_1|\le \frac{K}{\varepsilon} +(K\varepsilon+
\frac{C}{s_1})
\N,
\end{equation}
\begin{equation}\label{A20}
|A_2|+|A_3|+|A_4|
\le 
\frac{K}{\varepsilon} +K\varepsilon\N,
\end{equation}
\begin{equation}\label{A30}
|A_5|+|A_6|
\le C+ \frac{C}{ s_1}\N.
\end{equation}
\end{lem}

\bigskip

Indeed, from (\ref{et44}) and this Lemma,
we deduce that
\begin{align*}
\N\le \frac{K}{\varepsilon}
+(K\varepsilon+
 \frac{C}{s_1})
\N.
\end{align*}
Now, we can use the fact that
$s_1\ge -1 -\log(T^*(x)-t_3(x_0))\ge-1 -\log(T(x_0)-t_3(x_0))$ and we choose 
$T(x_0)-t_3(x_0)$  small enough, so that
$$ \frac{C}{s_1}\le \frac1{-1 -\log(T(x_0)-t_3(x_0))} \le C \varepsilon .$$
 If we choose  $\varepsilon$ small enough so that
$ \frac{C}{s_1}\le \frac14$ and $K \varepsilon \le \frac14$, we obtain
\begin{align*}
\N\le {K}.
\end{align*}
Since $[s,s+1]\subset  [s_1,s_2]$, we derive 
 (\ref{projan221}).

\bigskip

It remains to prove Lemma \ref{g}.

\bigskip

 Proof of Lemma \ref{g}: We first deal with the  estimate \eqref{control} and   \eqref{0control}.
First,  we divide $(-1,1)$ into two parts $A_{1}(s)$ and $A_{2}(s)$ defined in   \eqref{27nov1}.
 
\medskip

Note that, by using 
 the definition of the set $A_1(s)$ defined in \eqref{27nov1}  and the expression of $\p$ defined in \eqref{psi}, we get,  
\begin{equation}\label{fev3}
 | w(y,s)|^{p+1}\le  \p^{-\frac{p+1}2}(s)\le Ce^{-s}\le C,\quad \  \forall y\in A_1(s).
\end{equation}
 From  (\ref{equiv1}), \eqref{id1} and the expression of $\p$ defined in \eqref{psi}, 
this yields
\begin{equation}\label{123fev123}
e^{-\frac{2(p+1)s}{p-1}}s^{\frac{2a}{p-1}}   F(\p w)\le Ce^{-2s}+ \frac{C}{s^a}   |w(y,s )|^{p+1}\log^a (2+\p^2 w^2) \le C,\quad \  \forall y\in A_1(s).
\end{equation}

\medskip

Next, by using  the definition of the set $A_2(s)$ introduced in \eqref{27nov1},  the expression of $\p$ defined in \eqref{psi}  and   the estimate   \eqref{16dec5bis} proved in Section \ref{section2},  we conclude
\begin{equation} \label{fev4}
K^{-1}s\le  {\log (2+\p^2 w^2)}
 \le K  s.
\end{equation}
 From  (\ref{equiv1}), \eqref{id1} and \eqref{fev4}, 
this yields
\begin{equation}\label{fev1}
  \frac{C}{K}   |w(y,s )|^{p+1} \le C+C e^{-\frac{2(p+1)s}{p-1}}s^{\frac{2a}{p-1}}   F(\p w)\le C+ CK   |w(y,s )|^{p+1},\quad \  \forall y\in A_2(s).
\end{equation}
%
%
Adding  \eqref{fev3}, \eqref{123fev123} and  \eqref{fev1}, we conclude that  \eqref{fev1} is  still valid, for all $y\in (-1,1).$  Therefore,  the  estimates \eqref{control}   and \eqref{0control} follow  immediately from  \eqref{fev1}  after   integration over   $(-1,1).$
\medskip 

 Thanks to \eqref{control} and \eqref{0control},  we can adapt with no difficulty
 the proof  in the unperturbed case \cite{MZajm03,MZimrn05}  (up to some very minor changes),
in order to get   the proof of the estimates \eqref{control3}, \eqref{control1}, \eqref{control30}, \eqref{control4}  and \eqref{control5}. 
 Also, by using   \eqref{control} and   the Hardy inequality \eqref{Hardybis},  we easily   conclude \eqref{A10} and \eqref{A10}.

\medskip

Finally, it remains only to control the terms $A_5$ and $A_6$.
Note from  \eqref{equiv1},    \eqref{equiv2}  and  \eqref{equiv3}  that
\begin{equation}\label{fev6}
| F_1(\p w)|+| F_2(\p w)|\le   C+C \frac{ F(\p w)}{s}.
\end{equation}
The result \eqref{A30} follows  immediately from  \eqref{fev6}.
This concludes the proof of Lemma \eqref{g} and Proposition \eqref{projan22} too.
\Box

Proof of Theorem \ref{t2}:
Thanks to \eqref{projan221},  \eqref{et} and  \eqref{2018N00},  we deduce,  for all $s\ge  -\log (T^*(x)-t_2(x_0))$
\begin{equation}\label{fev7}
\int_{s}^{s+1}\iint \Big((\partial_{s}w)^2+(\partial_y w)^2(1-y^2)+w^2\Big)\w \y\t \le K.
\end{equation}
By using   the covering technique (we refer the reader to Merle
 and Zaag \cite{MZimrn05} (pure power case) and Hamza and Zaag 
\cite{HZjhde12}), we  conclude
\begin{equation}\label{fev8}
\int_{s}^{s+1}\iint \Big((\partial_{s}w)^2+(\partial_y w)^2+w^2\Big) \y\t \le K.
\end{equation}
Similarly  to  the proof of Proposition \ref{prop3.1} (Step 1),
we get
\begin{equation}\label{fev9}
\iint |w(y,s)|^{p+2} \y \le K.
\end{equation}
By   \eqref{fev9}, \eqref{0control} and Jensen's inequality, we infer
\begin{equation}\label{fev10}
\iint e^{-\frac{2(p+1)s}{p-1}}s^{\frac{2a}{p-1}} F(\p w){\mathrm{d}}y
\le   K.
\end{equation}
Finally,  the definition of  $L_0(w(s),s)$ given  in \eqref{5jan1} and the estimate \eqref{2018N0} imply
\begin{equation}\label{fev11}
\iint \Big((\partial_{s}w)^2+(\partial_y w)^2(1-y^2)+w^2\Big) \w \y \le K.
\end{equation}
%
%
%
%
Once again, by using   the covering technique,  we deduce \eqref{main1}.  This concludes the proof of Theorem \ref{t2}.
\Box




 \appendix

 \section{Some elementary lemmas.}
 Let $f$, $F$, $F_2$  be the functions defined in  \eqref{deff}, \eqref{defF}  and  \eqref{defF123}.  
Clearly, we have 
\begin{lem}\label{FFF} \ {\  }Let $q>1$,\\ 
\begin{align}
\int_0^u|v|^{q-1}v\log^{{a}}(2+v^2 )\v  \sim&  \frac{| u|^{q+1}}{q+1}\log^{{a}}(2+u^2  ),\quad  \text{ as } \;\; |u| \to \infty,
\label{estF0}\\
F(u)  \sim &\frac{uf(u)}{p+1} \quad  \text{ as } \;\; |u| \to \infty,\label{estF}\\
F_2(u)\sim &\frac{Cuf(u)}{\log^2(2+u^2)}\quad  \text{ as } \;\; |u| \to \infty.\label{estF3}
\end{align}
\end{lem}
\begin{proof} 
An integration by parts yields, for any $q>1$ and $a\in \er$,
\begin{equation}\label{IP0}
\int_0^u|v|^{q-1}v\log^{{a}}(2+v^2  )\v=
  \frac{| u|^{q+1}}{q+1}\log^{{a}}(2+u^2  )-\frac{2a}{q+1}{\int_0^u\frac{|v|^{q+1}v}{2+v^2}\log^{{a-1}}(2+v^2 )\v}.
\end{equation}
From the fact that,
$$\big|\frac{|v|^{q+1}v}{2+v^2}\log^{{a-1}}(2+v^2  )\big|\le C+C
|v|^{q}\log^{{a-1}}(v^2  + 2), \qquad \forall  v\in\er, $$
we can  write
\begin{align}\label{ddd}
|{\int_0^u\frac{|v|^{q+1}v}{2+v^2}\log^{{a-1}}(2+v^2  )\v}|\le C+C
  | u|^{q+1}\log^{a-1}(2+u^2  ),  \qquad \forall  u\in\er.
\end{align}
From \eqref{IP0} and \eqref{ddd}, one easily obtain
\begin{equation*}
\int_0^u|v|^{q-1}v\log^{{a}}(2+v^2  )\v  \sim  \frac{| u|^{q+1}}{q+1}\log^{{a}}(2+u^2 ),\quad  \text{ as } \;\; |u| \to \infty,
\end{equation*}
which ends the proof of  the estimates \eqref{estF0}.  Note that   \eqref{estF} is  trivial   from  \eqref{estF0} and  the definition of $f$ given  in  \eqref{defF}.\\
It remains to  prove  \eqref{estF3}.
Note that it easily follows from \eqref{IP0} that 
\begin{align}\label{IP1}
 F(u )=\int_0^uf(v)\v= & 
 \frac{| u|^{p+1}}{p+1}\log^{{a}}(2+u^2  )
-\frac{2a}{p+1}\int_0^u |v|^{p-1}v\log^{{a-1}}(2+v^2  )\v\nonumber\\
&+\frac{4a}{p+1}\int_0^u\frac{|v|^{p-1}v}{2+v^2}\log^{{a-1}}(2+v^2  )\v.
\end{align}
Once again, by integrating bt parts, we obtain 
\begin{align}\label{IP2}
\int_0^u |v|^{p-1}v\log^{{a-1}}(2+v^2  )\v=& \frac{| u|^{p+1}}{p+1}\log^{{a-1}}(2+u^2  )-\frac{2a-2}{p+1}\int_0^u |v|^{p-1}v\log^{{a-2}}(2+v^2  )\v\nonumber\\
&+\frac{4a-4}{p+1}\int_0^u\frac{|v|^{p-1}v}{2+v^2}\log^{{a-2}}(2+v^2  )\v.
\end{align}
Therefore,  \eqref{IP1}, \eqref{IP2}, \eqref{defF123} and \eqref{defF2}, imply that 
\begin{equation}\label{15jan1}
 F_2(u)= F^1_2(u)+ F^2_2(u),
\end{equation}
where
\begin{align}
 F^1_2(u)&=\frac{4a(a-1)}{(p+1)^2}\int_0^u |v|^{p-1}v\log^{{a-2}}(2+v^2  )\v,  \label{f33}\\
 F^2_2(u)&=\frac{4a}{p+1}\int_0^u\big(\frac{2a-2}{p+1}\frac{|v|^{p-1}v}{2+v^2}\log^{{a-2}}(2+v^2  )
-\frac{|v|^{p-1}v}{2+v^2}\log^{{a-1}}(2+v^2  )\big)\v.\label{15jan6}
\end{align}
Let  us find an equivalent to  $F_2$. 
By exploiting the following estimates
\begin{equation*}\label{r2}
\big|\frac{2a-2}{p+1}\frac{|v|^{p-1}v}{2+v^2}\log^{{a-2}}(2+v^2 )
-\frac{|v|^{p-1}v}{2+v^2}\log^{{a-1}}(2+v^2  )\big|\le C +C |v|^{p-\frac32},
\end{equation*}
one easily obtains 
\begin{equation}\label{15jan2}
\big|F^2_2(u)\big|\le C +C |u|^{p-\frac12}.
\end{equation}
The result  \eqref{estF3}   immediately  follows from \ref{estF0},\eqref{15jan1},  (\ref{f33}) and  \eqref{15jan2}, which ends the proof of Lemma \ref{FFF}.
\end{proof}

\Box

\bigskip


The following lemma  shows the asymptotic behavior of the solution of the
associated ODE 
\begin{equation}\label{A0}
v'' (t)= |v(t)|^{p-1}x \log^{a}(2+v^2(t)  ), \quad v(0)=A > 0\quad { \textrm and }\quad v'(0)=B > 0.
\end{equation}
 \begin{lem}\label{asym-psi-T}
The problem  \eqref{A0} has  one positive solution. Moreover,   there exist $\T <\infty$, such that the solution $\psi$ satisfies the following asymptotic:
 \begin{equation}\label{inequ-psi-ep-+}
 v (t) \sim \kappa_{a}(\T - t)^{-\frac{2}{p-1}} |\log(\T - t)|^{-\frac{a}{p-1}}, \text{ as } t \to\ T,
 \end{equation}
 where $\kappa_{a} =  \left(\frac{2^{1-2a}(p+1)}{(p-1)^{2-a}} \right)^{\frac1{p-1}}.$
 \end{lem}
 \begin{proof}
The uniqueness and local existence  of \eqref{A0} are derived by the Cauchy-Lipschitz property.  
Let $T$ be the maximum  time  of the existence of the positive solution, i.e.  $v(t)$ exists for all $t \in [0, T)$. We now prove that $T < + \infty$. By contradiction, we suppose that  the solution exists on $[0, + \infty)$. 
By multiplying equation $\eqref{A0}$ by $v'(t)$ and
integrating  with respect to time on $(0,t)$, we obtain 
\begin{equation}\label{A3}
(v'(t))^2=2F(v(t))+ C,
\end{equation}
where $F$ is defined in \eqref{defF}.
 Using 
\eqref{A0}, we conclude that 
 $v' $ is an increasing function, so for all $t\in [0,\infty)$ we have $v' (t)\geq v'(0)>0$. Then, \eqref{A3} becomes
\begin{equation}\label{A4}
v'(t)=\sqrt{2F(v(t))+ C}.
\end{equation}
Using the fact that $v'(t)> 0$ and $v(t) >0$, we deduce that 
 $$\lim_{t \to +\infty} \int_0^{t} \frac{v'(s)}{\sqrt{2F(v(s))+C}} \s = \lim_{t \to +\infty} \int_0^{t}\s = +\infty.$$  
Let us mention that  
\begin{equation}\label{equivF}
 F(v )\sim\frac{ \psi^{p+1}}{p+1}\log^{{a}}(v^2  + 2), \quad \textrm{as} \quad v\to \infty,
\end{equation}
  and  $\ds{\int_0^{t} \frac{v'(s)}{v^{\frac{p+1}2}(s)\log^{\frac{a}2}(v^2(s)  + 2)} \s}$  is bounded.  Thus,  the contradiction  follows. 

\medskip

 Let us now prove \eqref{inequ-psi-ep-+}. 
By integrating \eqref{A4} with respect to time $(0,t)$, we obtain 
\begin{equation}
T - t =  \int_{v(t)}^{+\infty}\frac{du}{\sqrt{ 2F(u(t))+C}}.
\end{equation}
Thus, for all $\delta \in (0,\frac{ p-1}2)$, there exist $t_\delta$ such that for all $t \in (t_\delta, T)$, we have
$$ \int_{v(t)}^{+\infty} \frac{du}{u^{\frac{p+1}2 + \delta}} \leq T -t \leq  \int_{v(t)}^{+\infty} \frac{du}{u^{\frac{p+1}2  -  \delta}}.$$
This implies  for all $t \in (t_\delta, T)$ that:
$$ C^{-1} (T - t)^{-\frac{1}{\frac{ p - 1}2 + \delta}} \leq v(t) \leq C  (T - t)^{-\frac{1}{ \frac{p - 1}2 - \delta}},$$
from which we have
$$\log v (t)  \sim - \frac{2}{p-1} \log(T  -t) \quad  \text{ as } \;\; t \to T,$$
and
\begin{equation}\label{15janv10}
\log (v^2 + 2) \sim -\frac{4}{p-1} \log(T-t) \quad \text{ as }\;\; t \to T.
\end{equation}
Hence, by using \eqref{A4}, \eqref{15janv10} and \eqref{equivF}, we obtain
\begin{equation}\label{equivpsi}
 \frac{v'(t)}{v^{\frac{p+1}2}(t)}  \sim   \sqrt{\frac{2}{p+1}} \left(\frac{4}{p-1}\right)^{\frac{a}2} |\log(T - t)|^{\frac{a}2},\quad \text{ as } \;\; t \to T.
\end{equation}
By integrating over $(t,T)$, we have 
\begin{align}\frac{2}{p-1} v^{\frac{1-p}2}(t) & \sim  \sqrt{\frac{2}{p+1}}\left(\frac{4}{p-1}\right)^{\frac{a}2} \int^T_t |\log(T - v)|^{\frac{a}2} \v\nonumber\\
& \sim \sqrt{\frac{2}{p+1}} \left( \frac{4}{p-1}\right)^{\frac{a}2} (T - t) |\log(T - t)|^{\frac{a}2} \quad \text{ as }\; t \to T. \label{equivpsip}
\end{align}
Using \eqref{equivpsip}, we see after
straightforward calculations that
$$ v (t)  \sim  \left(\frac{2(p+1)}{(p-1)^2} \right)^{\frac1{p-1}} \left( \frac{4}{p-1}\right)^{-\frac{a}{p-1}} (T - t)^{-\frac2{p-1}} |\log(T - t)|^{-\frac{a}{p-1}} \quad \text{ as }\; t \to T. $$
This concludes the proof of \eqref{inequ-psi-ep-+}.
\end{proof}

%
%

\bigskip

\bigskip

By integrating by parts (see Lemma \ref{FFF}), we can write  
\begin{equation}\label{5jan10}
uf(u)-(p+1)F(u)\sim  
\frac{2a}{p+1}|u|^{p+1}\log^{a-1}(2+u^2), \quad  {\textrm{as}} \ |u|\to \infty,
\end{equation} 
where $f$ and $F$ defined respectively in \eqref{deff} and \eqref{defF}.
 More precisely, we have for all $u\in \er$
\begin{equation}\label{5jan11}
\Big|uf(u)-(p+1)F(u)- 
\frac{2a}{p+1}|u|^{p+1}\log^{a-1}(2+u^2)\Big| \le C+C |u|^{p+1}\log^{a-2}(2+u^2).
\end{equation} 
Thanks to \eqref{5jan10} and \eqref{5jan11}, we will give the  first and the second order terms   in the   expansion of    the nonlinearity $F(x)$ defined in \eqref{defF},  when $|x|$ is  large enough.
More precisely,  we 
  now state the following estimates 
\begin{lem}\label{lemm:esth}
For all $s \geq 1$,  for all $z\in \er$,
\begin{align}
 C^{-1} \ps z f(\ps z))\le C+F\left(\ps z)\le C (1+\ps z f(\ps z)\right),\label{equiv1}\\
F(\ps z)\le C+ C |\ps z|^{\bar p +1},  \
\label{equiv4}\\
F_1(\ps z)\le C+C\frac{ \ps z}{s}f(\ps z),\label{equiv2}\\
F_2(\ps z)\le C+C \frac{ \ps z}{s^2}f(\ps z),\label{equiv3}
\end{align}
where
$\phi$,  $F$, $F_1$ and $F_2$  are given in  \eqref{defphi},    \eqref{defF}, 
 \eqref{defF2},   \eqref{defF123},  and 

\begin{equation}\label{defpb}
 \bar p =
\begin{cases}
p+1 & \text{ if } N=1,2, \\
p+\frac{2}{N-2}-\frac2{N-1} & \text{ if } N \ge 3.
\end{cases}
\end{equation}
\end{lem}
\begin{proof} Note that \eqref{equiv1} obviously follows from \eqref{estF}.
Similarly, by taking into account  the  inequality  $\log^a (2+u^2)\le C+  C |u|^{\bar p-p}$   and  \eqref{estF} we  conclude
\eqref{equiv4}.
In order to derive estimates \eqref{equiv2} and \eqref{equiv3}, considering the first case $z^2\ps \geq 4$, then the case $z^2\ps \leq 4$, we would obtain
 \eqref{equiv2} and \eqref{equiv3} by using \eqref{estF0}, \eqref{estF}  and\eqref{estF3}.
 This ends the proof of Lemma \ref{lemm:esth}.\Box
\end{proof}
\def\cprime{$'$} \def\cprime{$'$}
\providecommand{\bysame}{\leavevmode\hbox to3em{\hrulefill}\thinspace}
\providecommand{\MR}{\relax\ifhmode\unskip\space\fi MR }
\providecommand{\MRhref}[2]{%
  \href{http://www.ams.org/mathscinet-getitem?mr=#1}{#2}
}
\providecommand{\href}[2]{#2}


\noindent{\bf Address}:\\
  Imam Abdulrahman Bin Faisal University
P.O. Box 1982 Dammam, Saudi Arabia.\\
\vspace{-7mm}
\begin{verbatim}
e-mail:  mahamza@iau.edu.sa
\end{verbatim}
Universit\'e Paris 13, Institut Galil\'ee,
Laboratoire Analyse, G\'eom\'etrie et Applications, CNRS UMR 7539,
99 avenue J.B. Cl\'ement, 93430 Villetaneuse, France.\\
\vspace{-7mm}
\begin{verbatim}
e-mail: Hatem.Zaag@univ-paris13.fr
\end{verbatim}

\end{document}